\documentclass[11pt]{article}
    \usepackage{url}
    \usepackage{verbatim}
    \usepackage[titletoc]{appendix}
    \usepackage{graphicx}
	\usepackage{amsmath}
	\usepackage{amsthm}
	\usepackage{amsfonts}
	\usepackage{graphicx}
    \usepackage{nicefrac}
    \usepackage{longtable}
    \usepackage{color}
    \usepackage{graphicx, amssymb, subcaption,graphics}

\newtheorem{rem}{Remark}[section]
\newtheorem{thm}{Theorem}[section]

\newtheorem{lem}{Lemma}[section]

\newcommand{\nrm}[1]{\left\| #1 \right\|}
\newcommand{\ciptwo}[2]{\left\langle #1 , #2 \right\rangle_\Omega}

\newcommand\dt {\Delta t}

\newcommand{\eipvec}[2]{\left[ #1 , #2 \right]_{\Omega}}

	\newcommand\be {\begin{equation}}
	\newcommand\ee {\end{equation}}

	\newcommand\bx {{\bf x}}

	\title{Structure-preserving, energy stable numerical schemes for a liquid thin film coarsening model}

	\date{\today}

\begin{document}
	
	\author{
Juan Zhang \thanks{School of Mathematical Sciences, Beijing Normal University, Beijing 100875, P.R. China (201831130028@mail.bun.edu.cn)}
\and	
Cheng Wang\thanks{Department of Mathematics, The University of Massachusetts, North Dartmouth, MA  02747, USA (Corresponding Author: cwang1@umassd.edu)}
	\and
Steven M. Wise\thanks{Department of Mathematics, The University of Tennessee, Knoxville, TN 37996, USA (swise1@utk.edu)}
\and	
Zhengru Zhang\thanks{Laboratory of Mathematics and Complex Systems, Beijing Normal University, Beijing 100875, P.R. China (zrzhang@bnu.edu.cn)}
}


 	\maketitle
	\numberwithin{equation}{section}
	
\begin{abstract}	
In this paper, two finite difference numerical schemes are proposed and analyzed for the droplet liquid film model, with a singular Leonard-Jones energy potential involved. Both first and second order accurate temporal algorithms are considered. In the first order scheme, the convex potential and the surface diffusion terms are implicitly, while the concave potential term is updated explicitly. Furthermore, we provide a theoretical justification that this numerical algorithm has a unique solution, such that the positivity is always preserved for the phase variable at a point-wise level, so that a singularity is avoided in the scheme. In fact, the singular nature of the Leonard-Jones potential term around the value of 0 prevents the numerical solution reaching such singular value, so that the positivity structure is always preserved. Moreover, an unconditional energy stability of the numerical scheme is derived, without any restriction for the time step size. In the second order numerical scheme, the BDF temporal stencil is applied, and an alternate convex-concave decomposition is derived, so that the concave part corresponds to a quadratic energy. In turn, the combined Leonard-Jones potential term is treated implicitly, and the concave part the is approximated by a second order Adams-Bashforth explicit extrapolation, and an artificial Douglas-Dupont regularization term is added to ensure the energy stability. The unique solvability and the positivity-preserving property for the second order scheme could be similarly established.  In addition, optimal rate convergence analysis is provided for both the first and second order accurate schemes. A few numerical simulation results are also presented, 
which demonstrate the robustness of the numerical schemes.

	\bigskip



\end{abstract}
	

\section{Introduction}

Certain liquids on a solid, chemo-attractive substrate spontaneously form a droplet structure connected by a very thin precursor (or wetting) layer. After the droplets appear, coarsening will occur, whereby smaller droplets will shrink and larger droplets will grow. The coarsening behavior, especially the rate of coarsening, of droplets has been of great scientific interest. The average droplet size increases with the decrease of the number of droplets and the increase of the characteristic distance. The coarsening process can be mediated by two mechanisms: collapse or collision. Collapse occurs through mass exchange between droplets in the precursor/wetting layer, and the droplets can also drift and collide with each other like particles in the traditional Ostwald ripening process. Thin layers of viscous liquid are well described by the lubrication approximation, which capitalizes on the separation of horizontal and vertical length scales. Under the assumption that the liquid film does not evaporate, lubrication theory leads to a single equation for the height function, $\phi=\phi(x, t)>0$, of a time-dependent film~\cite{Oron97}, in the form of an $H^{-1}$ gradient flow:
	\begin{equation*}
\partial_{t}\phi = \nabla\cdot\left(\mathcal{M}(\phi)\nabla\delta_\phi F \right) .
	\end{equation*}
Here $F$ is the free energy of the film/substrate system and is given by 	
	\[
F(\phi)=\int_{\Omega}  \left( \frac12 | \nabla\phi |^2 +\mathcal{U}(\phi)\right) d \bx,
	\]
and the term $\delta_\phi F$ is the variational derivative of $F$. In more detail, the gradient term describes the linearized contribution of the liquid-air surface tension, while $\mathcal{U}(\phi)$ models the intermolecular forces between the substrate and the film. See the related research on the coarsening dynamics~\cite{Alexander1999, Alexander1997, Anna1997, Becker2003, Bertozzi2001, Glasner2008, Glasner2003, Kim2006, kohn02, Kohn2004, Limary2002, Limary2003, otto06, Pbgo1989, Pismen2004, Ravati2013, Thiele2001}.

In this paper, we consider a well-known Lennard-Jones-type potential~\cite{Israelachvili1992}, which is expressed as $\mathcal{U}(\phi)=\frac{1}{3}\phi^{-8}-\frac{4}{3}\phi^{-2}$. Other potentials are also physically relevant~\cite{Pismen2004}. The free energy is
	\begin{equation}
F(\phi)=\int_{\Omega} \left( \frac{1}{3} \phi^{-8} - \frac{4}{3}\phi^{-2} + \frac{\varepsilon^{2}}{2} | \nabla\phi |^2 \right) d \bx,
	\label{energy-AG-1}
	\end{equation}
where $\Omega\subset \mathbb{R}^3$ is a bounded domain. The  $H^{-1}$ gradient flow associated with the given free energy functional~\eqref{energy-AG-1} with constant mobility $\mathcal{M} \equiv1$ (non-constant mobility case could be handled in a similar way) is
	\begin{equation}
\partial_t \phi = \Delta \mu ,  \quad
\mu: = \delta_\phi F = - \frac83 ( \phi^{-9} - \phi^{-3} ) - \varepsilon^2 \Delta \phi .
	\label{equation-AG}
	\end{equation}
Mass is conserved, and, due to the gradient structure, the following energy dissipation law is formally available:
	\begin{equation*}
\frac{d}{dt}F(\phi(t))= - \int_{\Omega} \, |\nabla\mu|^2 \, d \bx \le 0 .
	\end{equation*}
For simplicity of presentation, we assume periodic boundary conditions hold over the rectangular domain $\Omega$. Other types of boundary conditions, such as homogeneous Neumann, can also be handled.

Some numerical simulation results have been reported for the droplet liquid film model~\cite{Glasner2008, Oron97}, while the numerical analysis for this model has been very limited. The highly nonlinear and highly singular nature of the Leonard-Jones potential makes the equation very challenging, at both the analytic and numerical levels. In particular, the structure of the potential function requires that the phase variable has to keep a fixed sign, i.e., being, either positive or negative, to avoid a singularity. A positivity-preserving structure is required for the numerical scheme and for physical reality. The theoretical justification of such a property, in addition to the energy dissipation property, turns out to be a major difficulty for the numerical analysis. Moreover, a convergence analysis for any numerical scheme for the liquid film equation~\eqref{equation-AG} has also been a long-standing open problem.

 In this paper we construct and analyze finite difference numerical schemes that preserve all these three important theoretical features, in both the first and second order temporal accuracy orders. In the first order scheme, we make use of the convex-concave decomposition of the physical energy, and come up with a semi-implicit numerical method. The positivity-preserving property will be theoretically established, which is based on the subtle fact that the numerical solution is equivalent to a minimization of a convex energy functional, and the singular nature of the $\phi^{-8}$ term around the value of $0$ prevents the numerical solution reaching the singular value. As a result, the positivity structure is always preserved, and the unique solvability analysis comes from the convex nature of the implicit parts in the scheme. The energy stability becomes a direct consequence of this convexity analysis. In the second order numerical scheme, we have to make use of an alternate convex-concave decomposition of the physical energy, by observing that $\frac{1}{3} \phi^{-8}-\frac{4}{3} \phi^{-2}+\frac{4}{3}A_{0} \phi^2$ is a convex function, provided $A_0$ is sufficiently large. In turn, the concave part becomes a quadratic energy, which is approximated by a second order Adams-Bashforth explicit extrapolation. The second order BDF temporal stencil is applied, and an artificial Douglas-Dupont regularization term is added to ensure energy stability.

In addition, a detailed convergence analysis of the proposed numerical schemes, including the first and second order accurate ones, will be derived in this paper. This analysis gives an optimal rate error estimate in the $\ell^\infty (0, T; H_h^{-1}) \cap \ell^2 (0, T; H_h^1)$. A key point in the analysis lies in the following fact: since the nonlinear potential term corresponds to a convex energy, the corresponding nonlinear error inner product is always non-negative. Furthermore, the error estimate associated with the surface diffusion term indicates an $\ell^2 (0,T; H_h^1)$ convergence. It will be the first such work for this model.

The rest of the paper is organized as follows. In Section~\ref{sec:finite-difference}, we describe the finite difference discretization of space and recall some basic facts. Section~\ref{sec:1st scheme}, we propose the first order numerical scheme and state the main theoretical results, including positivity-preserving, unique solvability, and unconditional energy stability of the scheme. Further, we  provide a detailed convergence analysis. The second order accurate scheme is proposed in  Section~\ref{sec:2nd scheme}; the unique solvability, the positivity-preserving property and an optimal rate convergence analysis are established as well, using similar ideas. Some numerical simulation results are presented in Section~\ref{sec:numerical results}. Finally, the concluding remarks are given in Section~\ref{sec:conclusion}.

	\section{Finite difference spatial discretization}
	\label{sec:finite-difference}

The model described above is inherently two-dimensional in nature. However, other models of this type may be three-dimensional. Our analyses to follow, therefore, will generically cover one through three space dimensions. For spatial approximation we use a staggered-grid finite difference approach. We use the notation and results for some discrete functions and operators from~\cite{guo16, wise10, wise09a}, \emph{et cetera}. Let $\Omega=(0,L_{x})\times(0,L_{y})\times(0,L_{z}),$ where for simplicity, we assume $L_{x}=L_{y}=L_{z}=:L>0$. Let $N\in\mathbb{N}$ be given, and define the grid spacing $h:=\frac{L}{N}$, i.e., a uniform mesh assumption is made  for simplicity of notation. We define the following two uniform, infinite grids with spacing $h>0$: $C:=\{p_{i} \, | \, i\in\mathbb{Z}\}$ and $E:=\{p_{i+1/2} \, | \, i\in\mathbb{Z}\}$, where $p_{i}=p(i):=(i-1/2)  h.$ Consider the following 3D discrete periodic function spaces:
	\begin{equation*}
	\begin{split}
\mathcal{C}_{\rm per} :=\left\{\upsilon:C\times C\times C\rightarrow\mathbb{R} \, \middle| \, \upsilon_{i,j,k}=\upsilon_{i+\alpha N,j+\beta N,k+\gamma N}, \ \forall \, i, j, k, \alpha, \beta, \gamma\in \mathbb{Z}, \right\},
	\\
\mathcal{E}_{\rm per}^{x}:= \left\{\upsilon:E\times C\times C\rightarrow\mathbb{R} \, \middle| \, \upsilon_{i+\frac{1}{2},j,k}=\upsilon_{i+\frac{1}{2}+\alpha N,j+\beta N,k+\gamma N}, \ \forall \, i, j, k, \alpha, \beta, \gamma\in \mathbb{Z} \right\} ,
	\end{split}
	\end{equation*}
with the identification $\upsilon_{i,j,k}=\upsilon(p_{i},p_{j},p_{k}),$ \emph{et cetera}. The spaces $\mathcal{E}_{\rm per}^{y}$ and $\mathcal{E}_{\rm per}^{z}$ are analogously defined. The functions of $\mathcal{C}_{\rm per}$ are called cell-centered functions, and the functions of $\mathcal{E}_{\rm per}^{x},$ $\mathcal{E}_{\rm per}^{y}$ and $\mathcal{E}_{\rm per}^{z}$ called face-centered functions. We also introduce the following mean zero space to facilitate the later analysis: $\mathring{\mathcal{C}}_{\rm per} :=\{\upsilon\in\mathcal{C}_{\rm per} \, | \, 0=\overline{\upsilon}:=\frac{h^{3}}{| \Omega |}\sum_{i,j,k=1}^{N}\upsilon_{i,j,k}\}$.
We define $\vec{\mathcal{E}}_{\rm per} :=\mathcal{E}_{\rm per}^{x}\times\mathcal{E}_{\rm per}^{y}\times\mathcal{E}_{\rm per}^{z}.$

The center-to-face difference and averaging operators, $A_{x},\;D_{x}:\mathcal{C}_{\rm per}\rightarrow\mathcal{E}_{\rm per}^{x}$, $A_{y},\;D_{y}:\mathcal{C}_{\rm per}\rightarrow\mathcal{E}_{\rm per}^{y}$, and $A_{z},\;D_{z}:\mathcal{C}_{\rm per}\rightarrow\mathcal{E}_{\rm per}^{z}$, are defined as follows:
	\begin{equation*}
	\begin{split}
A_{x}\upsilon_{i+\frac{1}{2},j,k}:=\frac{1}{2}(\upsilon_{i+1,j,k}+\upsilon_{i,j,k}),\;\;\;\;
D_{x}\upsilon_{i+\frac{1}{2},j,k}:=\frac{1}{h}(\upsilon_{i+1,j,k}-\upsilon_{i,j,k}),\\
A_{y}\upsilon_{i,j+\frac{1}{2},k}:=\frac{1}{2}(\upsilon_{i,j+1,k}+\upsilon_{i,j,k}),\;\;\;\;
D_{y}\upsilon_{i,j+\frac{1}{2},k}:=\frac{1}{h}(\upsilon_{i,j+1,k}-\upsilon_{i,j,k}),\\
A_{z}\upsilon_{i,j,k+\frac{1}{2}}:=\frac{1}{2}(\upsilon_{i,j,k+1}+\upsilon_{i,j,k}),\;\;\;\;
D_{z}\upsilon_{i,j,k+\frac{1}{2}}:=\frac{1}{h}(\upsilon_{i,j,k+1}-\upsilon_{i,j,k}).
	\end{split}
	\end{equation*}
Likewise, the face-to-center difference and averaging operators, $a_{x},\;d_{x}:\mathcal{E}_{\rm per}^{x}\rightarrow\mathcal{C}_{\rm per}$, $a_{y},\;d_{y}:\mathcal{E}_{\rm per}^{y}\rightarrow\mathcal{C}_{\rm per}$, and $a_{z},\;d_{z}:\mathcal{E}_{\rm per}^{z}\rightarrow\mathcal{C}_{\rm per}$, are defined as follows:
	\begin{equation*}
	\begin{split}
a_{x}\upsilon_{i,j,k}:=\frac{1}{2}(\upsilon_{i+\frac{1}{2},j,k}+\upsilon_{i-\frac{1}{2},j,k}),\quad
d_{x}\upsilon_{i,j,k}:=\frac{1}{h}(\upsilon_{i+\frac{1}{2},j,k}-\upsilon_{i-\frac{1}{2},j,k}),
	\\
a_{y}\upsilon_{i,j,k}:=\frac{1}{2}(\upsilon_{i,j+\frac{1}{2},k}+\upsilon_{i,j-\frac{1}{2},k}),\quad
d_{y}\upsilon_{i,j,k}:=\frac{1}{h}(\upsilon_{i,j+\frac{1}{2},k}-\upsilon_{i,j-\frac{1}{2},k}),
	\\
a_{z}\upsilon_{i,j,k}:=\frac{1}{2}(\upsilon_{i,j,k+\frac{1}{2}}+\upsilon_{i,j,k-\frac{1}{2}}),\quad
d_{z}\upsilon_{i,j,k}:=\frac{1}{h}(\upsilon_{i,j,k+\frac{1}{2}}-\upsilon_{i,j,k-\frac{1}{2}}).
	\end{split}
	\end{equation*}
The discrete gradient $\nabla_{h}:\mathcal{C}_{\rm per}\rightarrow\vec{\mathcal{E}}_{\rm per}$ and the discrete divergence $\nabla_{h}\cdot:\vec{\mathcal{E}}_{\rm per}\rightarrow\mathcal{C}_{\rm per}$ are defined via
	\begin{align*}
\nabla_{h}\upsilon_{i,j,k} &:= \left(D_{x}\upsilon_{i+\frac{1}{2},j,k},D_{y}\upsilon_{i,j+\frac{1}{2},k},D_{z}\upsilon_{i,j,k+\frac{1}{2}}\right),
	\\
\nabla_{h}\cdot \vec{f}_{i,j,k} & :=d_{x}f_{i,j,k}^{x}+d_{y}f_{i,j,k}^{y}+d_{z}f_{i,j,k}^{z},
	\end{align*}
where $\upsilon\in\mathcal{C}_{\rm per}$ and $\vec{f}=(f^{x},f^{y},f^{z})\in\vec{\mathcal{E}}_{\rm per}$. The standard 3D discrete Laplacian, $\Delta_{h}:\mathcal{C}_{\rm per}\rightarrow\mathcal{C}_{\rm per}$ is given by
	\begin{align*}
\Delta_{h}\upsilon_{i,j,k}:&=d_{x}(D_{x}v)_{i,j,k}+d_{y}(D_{y}v)_{i,j,k}+d_{z}(D_{z}v)_{i,j,k})
	\\
& = \frac{1}{h^{2}}(\upsilon_{i+1,j,k}+\upsilon_{i-1,j,k}+\upsilon_{i,j+1,k}+\upsilon_{i,j-1,k}+\upsilon_{i,j,k+1}+\upsilon_{i,j,k-1}-6\upsilon_{i,j,k}).
	\end{align*}
More generally, if $\mathcal{D}$ is a periodic scalar function that is defined at all of the face center points and $\vec{f}=(f^{x},f^{y},f^{z})\in\vec{\mathcal{E}}_{\rm per},$ then $\mathcal{D}\vec{f}=(f^{x},f^{y},f^{z})\in\vec{\mathcal{E}}_{\rm per},$ assuming point-wise multiplication, and we may define
	\begin{equation*}
\nabla_{h}\cdot (\mathcal{D}\vec{f})_{i,j,k}:=d_{x}(\mathcal{D}{f}^{x})_{i,j,k}
+d_{y}(\mathcal{D}{f}^{y})_{i,j,k}+d_{z}(\mathcal{D}{f}^{z})_{i,j,k}).
	\end{equation*}
Specifically, if $\upsilon\in\mathcal{C}_{\rm per},$ then $\nabla_{h}\cdot(\mathcal{D}\nabla_{h}):\mathcal{C}_{\rm per}\rightarrow\mathcal{C}_{\rm per}$ is defined point-wise via
	\begin{equation*}
\nabla_{h}\cdot(\mathcal{D}\nabla_{h}\upsilon)_{i,j,k}=d_{x}(\mathcal{D}D_{x}\upsilon)_{i,j,k}
+d_{y}(\mathcal{D}D_{y}\upsilon)_{i,j,k}+d_{z}(\mathcal{D}D_{z}\upsilon)_{i,j,k}.
	\end{equation*}

Now we are ready to  define the following grid inner products:
	\begin{align*}
\langle\upsilon,\xi\rangle_{\Omega} & := h^{3}\sum_{i,j,k=1}^{N}\upsilon_{i,j,k}\xi_{i,j,k}, \quad \upsilon,\xi\in\mathcal{C}_{\rm per},
	\\
[\upsilon,\xi]_{x} & :=\langle a_{x}(\upsilon\xi),1\rangle_{\Omega}, \quad \upsilon,\xi\in\mathcal{E}_{\rm per}^{x},
	\\
[\upsilon,\xi]_{y} & : =\langle a_{y}(\upsilon\xi),1\rangle_{\Omega}, \quad \upsilon,\xi\in\mathcal{E}_{\rm per}^{y},
	\\
[\upsilon,\xi]_{z} & :=\langle a_{z}(\upsilon\xi),1\rangle_{\Omega}, \quad \upsilon,\xi\in\mathcal{E}_{\rm per}^{z},
	\\
[\vec{f}_{1},\vec{f}_{2}]_{\Omega} & :=[f_{1}^{x},f_{2}^{x}]_{x}+[f_{1}^{y},f_{2}^{y}]_{y}+[f_{1}^{z},f_{2}^{z}]_{z}, \quad \vec{f}_{i}=(f_{i}^{x},f_{i}^{y},f_{i}^{z})\in\vec{\mathcal{E}}_{\rm per},\;i=1,2.
	\end{align*}
We define the following norms for grid functions: for $\upsilon\in\mathcal{C}_{\rm per},$ then $\| \upsilon \|_{2}^{2}:=\langle\upsilon,\upsilon\rangle_{\Omega};$ $\|\upsilon\|_{p}^{p}:=\langle|\upsilon|^{p},1\rangle_{\Omega},$ for $1\leq p <\infty,$ and $\| \upsilon \|_{\infty}:= \max_{1\leq i,j,k\leq N}  | \upsilon_{i,j,k} |$. We define norms of the gradient as follows: for $\upsilon\in\mathcal{C}_{\rm per},$
\begin{equation*}
\begin{split}
\|\nabla_{h}\upsilon\|_{2}^{2}:=[\nabla_{h}\upsilon,\nabla_{h}\upsilon]_{\Omega}=
[D_{x}\upsilon,D_{x}\upsilon]_{x}+[D_{y}\upsilon,D_{y}\upsilon]_{y}+[D_{z}\upsilon,D_{z}\upsilon]_{z},
\end{split}
\end{equation*}
and, more generally, for $1\leq p <\infty,$
\begin{equation*}
\begin{split}
\| \nabla_{h}\upsilon \|_{p}:=
([| D_{x}\upsilon|^{p},1]_{x}+[| D_{y}\upsilon|^{p},1]_{y}+[| D_{z}\upsilon|^{p},1]_{z})^{\frac{1}{p}}.
\end{split}
\end{equation*}
Higher order norms can be defined. For example,
\begin{equation*}
\| \upsilon \|_{H_{h}^{1}}^{2}:= \| \upsilon \|_{2}^{2}
 + \| \nabla_h \upsilon \|_{2}^{2}, \, \, \,
\| \upsilon \|_{H_{h}^{2}}^{2}:= \| \upsilon \|_{H_h^1}^{2}
 + \| \Delta_h \upsilon \|_2^2 .
\end{equation*}

	\begin{lem}[\cite{wise09a}]
	\label{L1}
Let $\mathcal{D}$ be an arbitrary periodic,\;scalar function defined on all of the face center points. For any $\psi,\nu\in\mathcal{C}_{\rm per}$ and any $\vec{f}\in\vec{\mathcal{E}}_{\rm per},$ the following summation by parts formulas are valid:
	\begin{equation*}
\langle\psi,\nabla_{h}\cdot\vec{f}\rangle_{\Omega}=-[\nabla_{h}\psi,\vec{f}]_{\Omega},\quad  \langle\psi,\nabla_{h}\cdot(\mathcal{D}\nabla_{h}\nu)\rangle_{\Omega} =-[\nabla_{h}\psi,\mathcal{D}\nabla_{h}\nu]_{\Omega}.
	\end{equation*}
	\end{lem}

To facilitate the convergence analysis, we need to introduce a discrete analogue of the space $H_{\rm per}^{-1},$ as outlined in \cite{wang11a}. For any $f\in\mathcal{C}_{\rm per}$, with $\overline{f}:= |\Omega|^{-1}\langle f , 1 \rangle_\Omega =0$, we define $\psi:= (-\Delta_h)^{-1} f\in\mathcal{C}_{\rm per}$ as the periodic solution of
	\begin{equation*}
- \Delta_h \psi = f ,  \quad \overline{\psi} = 0 .
	\end{equation*}
In turn, the inner product $\langle \, \cdot , \, \cdot \, \rangle_{-1,h}$ and the $\| \cdot \|_{-1,h}$ norm are introduced as
	\begin{equation*}
  \langle f ,g \rangle_{-1,h} := \langle f , (-\Delta)^{-1} g \rangle_\Omega
  = \langle (-\Delta)^{-1} f, g \rangle_\Omega ,  \quad
  \| f \|_{-1,h} := \sqrt{ \langle f , f \rangle_{-1,h} } .
	\end{equation*}

	\begin{lem}[\cite{cheng16a}]
	\label{lem:L4 est}
For any $f \in \mathring{\mathcal{C}}_{\rm per}$, we have
\begin{equation*}
 \| f \|_4 \le C \| f \|_{-1,h}^\frac18 \cdot \| \nabla_h f \|_2^\frac78 ,
	\end{equation*}
for some constant $C>0$ that is independent of $h$ and $f$.
	\end{lem}
	
To prove the positivity of our methods, we will need the following result.

	\begin{lem}[\cite{chen19b}]
	\label{CH-positivity-Lem-0}
Suppose that $\phi_{1},\phi_{2}\in\mathcal{C}_{\rm per},$ with $\phi_{1}-\phi_{2}\in\mathring{\mathcal{C}}_{\rm per}$. Assume that $0<\phi_{1},\phi_{2}<M_{h}$, for any  $M_h>0$ that may depend on $h$. Then we have the following estimate:
	\begin{equation}
\| (- \Delta_h)^{-1} ( \phi_1 - \phi_2 ) \|_\infty \le C M_h ,
 	\label{AG-Lem-0}
	\end{equation}
for some constant $C>0$ that only depends on $\Omega$.
	\end{lem}

	\section {A first order numerical scheme}
	\label{sec:1st scheme}
We follow a convex-concave decomposition methodology in deriving our scheme, and consider the following semi-implicit, fully discrete schemes: given $\phi^{n}\in\mathcal{C}_{\rm per},$ find $\phi^{n+1}, \mu^{n+1}\in\mathcal{C}_{\rm per},$ such that
	\begin{equation}
\frac{\phi^{n+1}-\phi^{n}}{\dt}=\Delta_h \mu^{n+1},
  \quad
  \mu^{n+1}=-\frac{8}{3}(\phi^{n+1})^{-9} + \frac{8}{3}(\phi^{n})^{-3}
  - \varepsilon^2  \Delta_h\phi^{n+1} .
	\label{scheme-AG-1}
	\end{equation}
If solutions exist, it is obvious that the numerical scheme  is mass conservative, i.e., $\overline{\phi^{n+1}} = \overline{\phi^n} = \cdots = \overline{\phi^0} := \beta_0$, since
	\[
\ciptwo{\phi^{n+1}-\phi^{n}}{1} = 	\dt \ciptwo{\Delta_h \mu^{n+1}}{1} = -\dt\eipvec{\nabla_h\mu^{n+1}}{\nabla_h 1} = 0.
	\]
We will describe in a later section how to construct the initial data for the discrete scheme, $\phi^0$, from the initial data  for the PDE system~\eqref{equation-AG}.
	
We define the discrete energy $F_{h}(\phi):\mathcal{C}_{\rm per}\to \mathbb{R}$ as
	\begin{equation} \label{def: discrete energy}
F_{h}(\phi) = F_{h,c}(\phi)-F_{h,e}(\phi) , \quad F_{h,c}(\phi)=\frac{1}{3} \langle \phi^{-8} , 1 \rangle_\Omega + \frac{\varepsilon^{2}}{2} \| \nabla_h \phi \|_2^2 , \quad F_{h,e}(\phi)=\frac{4}{3} \langle  \phi^{-2} , 1 \rangle_\Omega .
	\end{equation}
Observe that both $F_{h,c}(\phi)$ and $F_{h,e}(\phi)$ are convex, and both require $\phi$ to be positive to be well defined.  Next, we prove that positive numerical solutions for~\eqref{scheme-AG-1} exist and are unique.

	\subsection{Existence and uniqueness of positive solutions}
	
	\begin{thm}
	\label{CH-positivity}
Let $\phi^{n}\in\mathcal{C}_{\rm per}$, with $\delta_0  \le \phi^n$, point-wise, for some $\delta_0  >0$. Then, there exists a unique solution $\phi^{n+1}\in\mathcal{C}_{\rm per}$ to the scheme~\eqref{scheme-AG-1}, with $\overline{\phi^{n+1}}=\overline{{\phi}^n},$ and $\phi^{n+1} >0$, point-wise.
	\end{thm}

	\begin{proof}
If a positive numerical solution exists, then it must satisfy the mass conservation property. Thus, we will only look for such solutions. Set $\beta_0 := \overline{\phi^n}$ and $\varphi^n := \phi^n - \beta_0$. The numerical solution of~\eqref{scheme-AG-1} must be a minimizer of the following discrete energy functional:
	\begin{equation}
\mathcal{J}(\varphi):=\frac{1}{2\dt} \| \varphi-\varphi^{n} \|_{-1,h}^{2} + \frac{1}{3}\langle(\varphi+\beta_0)^{-8},1\rangle_{\Omega} + \frac{\varepsilon^{2}}{2} \| \nabla_h \varphi \|_2^2 + \frac{8}{3} \langle (\varphi + \beta_0), (\phi^{n})^{-3} \rangle_{\Omega} ,
	\label{AG-positive-1}
	\end{equation}
over the admissible set
	\begin{equation}
\mathring{V}_{h}:=\{\varphi \in \mathring{\mathcal{C}}_{\rm per} \, | \,  -\beta_0 <\varphi \leq M_{h}-\beta_0 \}\subset\mathbb{R}^{N^{3}},
	\label{def-V_h}
	\end{equation}
where $M_{h}=\frac{ \beta_0 |\Omega| }{h^{3}}$. In fact, such a bound comes from the fact that, the numerical value of $\varphi$ at any numerical grid point $(i,j,k)$ has to be bounded by $M_h - \beta_0$, due to the fact that $\sum_{i,j,k}^N \phi_{i,j,k} = \frac{ \beta_0 |\Omega| }{h^{3}}$, as well as the positivity assumption for $\phi$: $\phi_{i,j,k} > 0$, for any $(i,j,k)$. The function $\mathcal{J}$ is a strictly convex function over this domain. Consider the following closed domain: for any $\delta\in (0,1)$,
	\begin{equation}
\mathring{V}_{\delta,h}:=\{\varphi\in\mathring{\mathcal{C}}_{\rm per}\, | \, \delta - \beta_0 \le \varphi \le M_{h}  - \beta_0 \}\subset\mathbb{R}^{N^{3}}.
 	\label{def-V_h-delta}
	\end{equation}
Since $\mathring{V}_{\delta,h}$ is a bounded, compact, and convex set in the subspace $\mathring{\mathcal{C}}_{\rm per},$ there exists a (not necessarily unique) minimizer of $\mathcal{J}$ over $\mathring{V}_{\delta,h}.$ The key point of the positivity analysis is that such a minimizer can not occur on the left boundary, when $\delta$ is sufficiently small.

To get a contradiction, assume that the minimizer of $\mathcal{J}$, denoted $\varphi^\star$, occurs at the left boundary of $\mathring{V}_{\delta,h}$. In particular, assume that
	\[
\varphi_{\vec{\alpha_{0}}}^{\star}=\delta - \beta_0,
	\]
where $\vec{\alpha_{0}} := (i_0,j_0,k_0)$. Thus, the grid function $\varphi^{\star}$ has a global minimum at the grid point $\vec{\alpha_{0}}$. Suppose that $\varphi^\star$ reaches its global maximum at the grid point $\vec{\alpha_{1}}:=(i_1,j_1,k_1)$. By the the fact that $\overline{\varphi^{\star}}=0$, it is obvious that
$M_{h} \ge \varphi_{\vec{\alpha_{1}}}^{\star} + \beta_0 \ge \beta_0$.

The function $\mathcal{J}$ is smooth over $\mathring{V}_{\delta,h}$, and, for any $\psi\in\mathring{\mathcal{C}}_{\rm per}$, the directional derivative is precisely
	\begin{equation}
	\begin{split}
d_{s}\mathcal{J}(\varphi^{\star}+s\psi) |_{s=0} & := \frac{1}{\dt} \langle(-\Delta_h)^{-1} (\varphi^{\star} -\varphi^{n}), \psi \rangle_\Omega - \frac{8}{3} \langle (\varphi^{\star} + \beta_0)^{-9}, \psi \rangle_{\Omega}
	\\
& \quad  - \varepsilon^2 \langle \Delta_h \varphi^{\star} , \psi \rangle_{\Omega} + \frac{8}{3} \langle (\phi^n)^{-3},\psi \rangle_{\Omega} .
	\end{split}
	\end{equation}
Let us pick the particular direction $\psi \in \mathring{\mathcal{C}}_{\rm per}$, satisfying
	\begin{equation*}
\psi_{i,j,k} =\delta_{i,i_0}\delta_{j,j_0}\delta_{k,k_0}-\delta_{i,i_1}\delta_{j,j_1}\delta_{k,k_1},
	\end{equation*}
where $\delta_{i,j}$ is the Dirac delta function. Then the derivative may be expressed as
	\begin{equation}
	\label{AG-positive-4}
	\begin{split}
\frac{1}{h^{3}}d_{s}\mathcal{J}(\varphi^{\star}+s\psi)|_{s=0} := & \frac{1}{\dt} (-\Delta_h)^{-1} (\varphi^*_{\vec{\alpha_{0}}} + \beta_0  - \phi^{n}_{\vec{\alpha_{0}}}) - \frac{1}{\dt} (-\Delta_h)^{-1} (\varphi^{\star}_{\vec{\alpha_{1}}} + \beta_0 - \phi^n_{\vec{\alpha_{1}}})
	\\
& \quad -\frac{8}{3} (\varphi^{\star}_{\vec{\alpha_{0}}} + \beta_0)^{-9} +\frac{8}{3} (\varphi^{\star}_{\vec{\alpha_{1}}} + \beta_0)^{-9}
	\\
& \quad +\frac{8}{3} (\phi^{n}_{\vec{\alpha_{0}}} )^{-3} - \frac{8}{3} (\phi^{n}_{\vec{\alpha_{1}}} )^{-3} - \varepsilon^2 ( \Delta_h \varphi^{\star}_{\vec{\alpha_{0}}} - \Delta_h \varphi^{\star}_{\vec{\alpha_{1}}} ) .
	\end{split}
	\end{equation}
Let us define $\phi^{\star}:=\varphi^{\star} + \beta_0$. Because $\phi^{\star}_{\vec{\alpha_{0}}}=\delta$ and $\phi^{\star}_{\vec{\alpha_{1}}} \ge \beta_0$, we have
	\begin{equation}
-\frac{8}{3}(\phi^{\star}_{\vec{\alpha_{0}}})^{-9}+\frac{8}{3}(\phi^{\star}_{\vec{\alpha_{1}}})^{-9} \leq-\frac{8}{3}(\delta^{-9}-\beta_0^{-9}) .
	\label{AG-positive-5}
	\end{equation}
Since $\phi^{\star}$ has a global minimum at the grid point $\vec{\alpha_{0}}$, with, $\phi^{\star}_{\vec{\alpha_{0}}}=\delta\leq\phi^{\star}_{i,j,k}$, and a maximum at the grid point $\vec{\alpha_{1}}$, with $\phi^{\star}_{\vec{\alpha_{1}}}\geq\phi^{\star}_{i,j,k}$, for any $(i,j,k)$, the following results are valid:
\begin{equation}
\Delta_h \varphi^{\star}_{\vec{\alpha_{0}}} \ge 0, \quad
\Delta_h \varphi^{\star}_{\vec{\alpha_{1}}} \le 0.  \label{AG-positive-6}
\end{equation}
For the numerical solution $\phi^n$ at the previous time step, the a-prior assumption $\delta_0 \leq\phi^{n}\leq M_h$ indicates that
	\begin{equation}
- \delta_0^{-3}<(\phi^{n}_{\vec{\alpha_0}})^{-3} - (\phi^{n}_{\vec{\alpha_{1}}})^{-3}< \delta_0^{-3}.
	\label{AG-positive-8}
	\end{equation}
For the first two terms appearing in~\eqref{AG-positive-4}, we apply Lemma~\ref{CH-positivity-Lem-0} and obtain
	\begin{equation}
- 2 C M_h \leq(-\Delta_{h})^{-1}(\phi^{\star}-\phi^{n})_{\vec{\alpha_{0}}} -(-\Delta_{h})^{-1}(\phi^{\star}-\phi^{n})_{\vec{\alpha_{1}}}\leq 2 C M_h.
	\label{AG-positive-9}
	\end{equation}
Consequently, a substitution of (\ref{AG-positive-5}) -- (\ref{AG-positive-9}) into (\ref{AG-positive-4}) yields
	\begin{equation}
\frac{1}{h^{3}}d_{s}\mathcal{J} (\varphi^{\star}+s\psi)|_{s=0} \le -\frac{8}{3}(\delta^{-9}-\beta_0^{-9}) +\delta_0^{-3} + 2 C \dt^{-1} M_h .
	\label{AG-positive-10}
	\end{equation}
Define $D_0 := \delta_0^{-3}  + 2 C \dt^{-1} M_h$, and note that $D_0$ is a constant for fixed $\dt$ and $h$, even though it becomes singular as $\dt, h\to 0$. In particular, $D_0 = O (\dt^{-1} h^{-2}  + \delta_0^{-3})$. In any case, for any fixed $\dt$ and $h$, we can choose $\delta$ small enough so that
	\begin{equation}
- \frac83 ( \delta^{-9} - \beta_0^{-9}  ) + D_0 < 0 .	\label{AG-positive-11}
	\end{equation}
This shows that
	\begin{equation}
\frac{1}{h^{3}}d_{s}\mathcal{J} (\varphi^{\star}+s\psi)|_{s=0} < 0 ,
	\label{AG-positive-12}
	\end{equation}
for $\delta$ satisfying~\eqref{AG-positive-11}. This contradicts the assumption that $\mathcal{J}$ has a minimum at $\varphi^{\star}$.  Therefore, the global minimum of $\mathcal{J}(\varphi)$ over $\mathring{V}_{\delta,h}$ could only possibly occur at interior point, if $\delta >0$ is small. We conclude that there must be a solution $\varphi\in\mathring{\mathcal{C}}_{\rm per}$, so that
	\begin{equation}
\frac{1}{h^{3}}d_{s}\mathcal{J}(\varphi+s\psi)|_{s=0}=0,
	\end{equation}
such that $\phi=\varphi + \beta_0$ is positive. Such a minimizer is equivalent to the numerical solution of~\eqref{scheme-AG-1}. The existence of a positive numerical solution is established.

In addition, since $\mathcal{J}(\varphi)$ is a strictly convex function over $V_{h}$, the uniqueness analysis for this numerical solution is straightforward. The proof of Theorem~\ref{CH-positivity} is complete.
	\end{proof}

	\begin{rem}
This positivity-preserving analysis has been very successfully applied to the Cahn-Hilliard equation with Flory-Huggins-type logarithmic potentials~\cite{chen19b, dong19a, dong20a}. It is the first time, to our knowledge, that it has been applied to a gradient flow with singular Leonard-Jones style potential. Also see related works~\cite{Du2019, Du2020a} for the related maximum bound principles for a class of semi-linear parabolic equations.
	\end{rem}

\begin{rem}
In the positivity analysis, it is observed that $\phi^{n+1} \ge \delta$, while $\delta$ depends on $\dt$ and $h$ in a form of $\delta = O ( \dt^\frac19 + h^\frac29)$, so that $\delta \to 0$ as $\dt, h \to 0$. in other words, a directly numerical analysis would not be able to ensure a uniform distance between the numerical solution and the singular limit value of $0$. On the other hand, in the 1-D and 2-D cases, the phase separation is expected to be available at the PDE analysis level,  i.e., a uniform distance between the phase variable and the singular limit value $0$, with such a uniform distance dependent on $\varepsilon$, $\theta_0$ and the initial data. Such a theoretical property has been established for the Cahn-Hilliard model with logarithmic Flory-Huggins energy potential; see the related existing works~\cite{miranville11, Giorgini17a, miranville12}, etc. With such a theoretical bound available for the PDE solution, we are also able to derive a similar bound for the numerical solution, in combination with the convergence analysis and error estimate presented in the later section.
\end{rem}

Because of the convex-concave decomposition structure of the numerical scheme, an  unconditional energy stability is available. 

\begin{thm} \label{thm:energy stability-1st}
The scheme~\eqref{scheme-AG-1} is unconditional energy stable: for any time step $\dt>0$,
\begin{equation*}
\begin{split}
F_{h}(\phi^{k+1})+\dt\|\nabla_{h}\mu^{k+1}\|_{2}^{2}\;\leq F_{h}(\phi^{k}).
\end{split}
\end{equation*}
\end{thm}

\begin{proof}
The following estimate is  valid:
	\begin{equation*}
	\begin{split}
F_{h}(\phi^{k+1})-F_{h}(\phi^{k})&=F_{h,c}(\phi^{k+1})-F_{h,e}(\phi^{k+1})-F_{h,c}(\phi^{k})+F_{h,e}(\phi^{k})
	\\
&=F_{h,c}(\phi^{k+1})-F_{h,c}(\phi^{k})-(F_{h,e}(\phi^{k+1})-F_{h,e}(\phi^{k}))
	\\
&\leq\langle\delta_{\phi}F_{h,c}(\phi^{k+1}),\phi^{k+1}-\phi^{k}\rangle_{\Omega} -\langle\delta_{\phi}F_{h,e}(\phi^{k}),\phi^{k+1}-\phi^{k}\rangle_{\Omega}
	\\
&=\langle\delta_{\phi}F_{h,c}(\phi^{k+1})-\delta_{\phi}F_{h,e}(\phi^{k}),\phi^{k+1}-\phi^{k}\rangle_{\Omega}
	\\
&=\langle\mu^{k+1},\phi^{k+1}-\phi^{k}\rangle_{\Omega}
	\\
&=\dt \langle\mu^{k+1},\Delta_{h}\mu^{k+1}\rangle_{\Omega} =-\dt\|\nabla_{h}\mu^{k+1}\|_{2}^{2} \le 0 .
	\end{split}
	\end{equation*}
	\end{proof}

\subsection {The $\ell^{\infty}(0,T; H_h^{-1}) \cap \ell^{2}(0,T; H_h^1)$ convergence analysis}

By $\Phi$ we denote the exact solution for the PDE~\eqref{equation-AG}. Define $\Phi_N ( \, \cdot \, ,t):=\mathcal{P}_{N}\Phi( \, \cdot \, ,t)$, where $\mathcal{P}_{N}:L^2(\Omega) \to\mathbb{T}_K^d$ is the (spatial) Fourier projection into the space of trigonometric polynomials of degree up to and including $K$, and $N=2K+1$. In particular, in 3D
	\[
\mathcal{P}_{N}f({\bf x}) = \sum_{\nrm{{\bf k}}_1\le K}\hat{f}_{{\bf k}}\exp\left(\frac{2\pi \mathfrak{i}}{L} \, {\bf k}\cdot{\bf x}\right), \quad \hat{f}_{{\bf k}} = \frac{1}{L^3}\int_{\Omega} f(x) \exp\left(-\frac{2\pi \mathfrak{i}}{L} \, {\bf k}\cdot{\bf x}\right) d{\bf x}.
	\]
The following projection approximation is standard: if $\Phi\in L^{\infty}(0,T;H_{\rm per}^{\ell}(\Omega))$, for some $\ell\in \mathbb{N}$,
	\begin{equation}
\|\Phi_{N}-\Phi\|_{L^{\infty}(0,T;H^{k})}\leq Ch^{\ell-k}\|\Phi\|_{L^{\infty}(0,T;H^{\ell})}, \quad \forall\, k: \, 0\leq k\leq \ell.
	\end{equation}
By $\Phi_{N}^{m},\Phi^{m}$ we denote $\Phi_{N}( \, \cdot \, ,t_{m}),$ and $\Phi( \, \cdot \, ,t_{m})$, respectively, with $t_{m}=m\cdot\dt$. Since $\Phi_{N}\in\mathbb{T}_K^d$, the mass conservative  property is available at the time-discrete level:
	\begin{equation}
\overline{\Phi_{N}^{m}}=\frac{1}{L^3}\int_{\Omega}\Phi_{N}( \, \cdot \, ,t_{m}) \, d{\bf x}
=\frac{1}{L^3}\int_{\Omega}\Phi_{N}( \, \cdot \, ,t_{m-1}) \, d{\bf x}=\overline{\Phi_{N}^{m-1}},\quad \forall \, m\in\mathbb{N}.
	\end{equation}
As shown before, we know that the numerical solution~\eqref{scheme-AG-1} is mass conservative at the fully discrete level:
	\begin{equation}
\overline{\phi^{m}}=\overline{\phi^{m-1}},\quad \forall \, m\in\mathbb{N}.
	\end{equation}
	
Define the grid projection operator, $\mathcal{P}_{h}: C^0_{\rm per}(\Omega)\to \mathcal{C}_{\rm per}$, by $\mathcal{P}_{h}f_{i,j,k} =  f(p_i,p_j,p_k)$, for all $f\in C^0_{\rm per}(\Omega)$. For the initial data, we take
	\begin{equation}
\phi^{0}=\mathcal{P}_{h}\Phi_{N}( \, \cdot \, ,t=0).
	\label{initial data-1}
	\end{equation}
The error grid function is defined as
	\begin{equation}
\tilde{\phi}_{m}:=\mathcal{P}_{h}\Phi^{m}_{N}-\phi^{m}, \quad \forall \, m\in\{0,1,2,3,\cdots\}.
	\end{equation}
Because the initial data are chosen using composition of operators $\mathcal{P}_h\mathcal{P}_N$, it is not hard to show that
	\[
\overline{\tilde{\phi}_{m}} := \frac{1}{L^3}\ciptwo{\tilde{\phi}_{m}}{1} = 0, \quad  \forall \, m\in\{0,1,2,3, \cdots \},
	\]
so that the discrete norm $\| \, \cdot \, \|_{-1,h}$ is well defined for the error grid function.

In the following theorem, the convergence result for the numerical scheme~\eqref{scheme-AG-1} is stated.

	\begin{thm}
	\label{thm:convergence-1st}
Given initial data $\Phi( \, \cdot \, ,t=0)\in C^{6}_{\rm per}(\Omega),$ suppose the solution of the PDE~\eqref{equation-AG}, $\Phi$, is in the regularity class
	 \[
\mathcal{R}:=H^{2}(0,T;C_{\rm per}(\Omega)) \cap H^1 (0,T;C^{2}_{\rm per} (\Omega)) \cap L^{\infty}(0,T;C^{6}_{\rm per}(\Omega)).
	\]
Then, provided $\dt$ and $h$ are sufficiently small, for all positive integers n, such that $t_{n}\leq T,$ we have the following convergence estimate for the numerical solution
	\begin{equation}
\|\tilde{\phi}^{n}\|_{-1,h}+\left(\varepsilon^{2} \dt \sum_{m=1}^{n}\|\nabla_{h}\tilde{\phi}^{m}\|_{2}^{2}\right)^{\frac{1}{2}}\leq C(\dt+h^{2}),
	\end{equation}
where $C>0$ is independent of $n$, $\dt$ and $h$.
	\end{thm}

\begin{proof}
A careful consistency analysis indicates the following truncation error estimate:
\begin{equation}
\frac{\Phi^{n+1}_{N}-\Phi^{n}_{N}}{\dt} = \Delta_{h} \Big( -\frac{8}{3}(\Phi^{n+1}_{N})^{-9}+\frac{8}{3}(\Phi^{n}_{N})^{-3} -\varepsilon^2 \Delta_{h} \Phi^{n+1}_{N} \Big) + \tau^{n},
  \label{AG-consistency-1}
\end{equation}
with $\|\tau^{n}\|_{-1,h}\le C(\dt+h^{2}).$ Observe that in equation~\eqref{AG-consistency-1}, and from this point forward, we drop the operator $\mathcal{P}_{h}$, which should appear in front of $\Phi_{N}$, for simplicity. In turn, subtracting the numerical scheme~\eqref{scheme-AG-1} from~\eqref{AG-consistency-1} gives
\begin{equation}
\frac{\tilde{\phi}^{n+1} - \tilde{\phi}^n}{\dt} =\Delta_h
\left( -\frac{8}{3} (\Phi^{n+1}_N)^{-9} + \frac{8}{3} (\phi^{n+1})^{-9}
+ \frac{8}{3} (\Phi^{n}_N)^{-3} - \frac{8}{3} (\phi^{n})^{-3}
- \varepsilon^2 \Delta_h \tilde{\phi}^{n+1}\right) + \tau^{n} .  \label{AG-consistency-2}
\end{equation}
Taking the discrete inner product of~\eqref{AG-consistency-2} with $2(-\Delta_{h})^{-1} \tilde{\phi}^{n+1}$ yields
	\begin{equation}
	\label{AG-convergence-1}
	\begin{split}
& \langle\tilde{\phi}^{n+1}-\tilde{\phi}^{n}, 2(-\Delta_{h})^{-1} \tilde{\phi}^{n+1} \rangle_{\Omega} + 2 \varepsilon^2 \dt \langle \Delta_h^2 \tilde{\phi}^{n+1},
 (-\Delta_h)^{-1} \tilde{\phi}^{n+1} \rangle_{\Omega}  \\
& \quad =  \frac{16}{3} \dt \langle (\Phi^{n+1}_N)^{-9} - (\Phi^{n}_N)^{-3} - (\phi^{n+1})^{-9}
 +(\phi^{n})^{-3} , \tilde{\phi}^{n+1} \rangle_{\Omega}
 + 2 \dt \langle \tau^{n},  \tilde{\phi}^{n+1}\rangle_{-1,h} .
	\end{split}
	\end{equation}
For the time difference, the following identity is valid:
	\begin{equation}
 2 \langle\tilde{\phi}^{n+1}-\tilde{\phi}^{n},  (-\Delta_{h})^{-1} \tilde{\phi}^{n+1}\rangle_{\Omega} =\|\tilde{\phi}^{n+1}\|_{-1,h}^{2}-\|\tilde{\phi}^{n}\|_{-1,h}^{2} +\|\tilde{\phi}^{n+1}-\tilde{\phi}^{n}\|_{-1,h}^{2} .
	\label{AG-convergence-2}
	\end{equation}
The estimate for the surface diffusion term is straightforward:
\begin{equation}
\langle \Delta_{h}^{2} \tilde{\phi}^{n+1},
 (-\Delta_{h})^{-1} \tilde{\phi}^{n+1}\rangle_{\Omega}
 = \| \nabla_h \tilde{\phi}^{n+1} \|_{2}^{2} .
   \label{AG-convergence-3}
\end{equation}
For the nonlinear inner product, the convex nature of $\frac{1}{3}\phi^{-8}$ implies the following result:
\begin{equation}
  \langle (\Phi^{n+1}_N)^{-9} - (\phi^{n+1})^{-9},\tilde{\phi}^{n+1} \rangle_{\Omega} \le 0.
  \label{AG-convergence-4}
\end{equation}
For the concave nonlinear inner product, we begin with an application of mean value theorem:
	\begin{equation}
(\Phi^{n}_{N})^{-3} - (\phi^{n})^{-3} = -3 ( \xi^n )^{-4} \tilde{\phi}^n ,
	\label{AG-convergence-5-1}
	\end{equation}
where $\xi^n\in\mathcal{C}_{\rm per}$ is a grid function that is between $\Phi^{n}_{N}$ and $\phi^n$ in a point-wise sense. Moreover, we see that
\begin{eqnarray}
   \| ( \xi^n )^{-4} \|_2 \le \max \left( \| (\Phi^{n}_N)^{-1} \|_8^4 , \| ( \phi^n )^{-1} \|_8^4 \right)
   \le C^* ,  \label{AG-convergence-5-2}
\end{eqnarray}
which is based on the smoothness of the projection solution $\Phi_N$, as well as the $\| \cdot \|_8$ bound of the numerical solution $(\phi^n)^{-1}$, coming from the energy stability. Then we arrive at
	\begin{align}
-\frac{16}{3} \langle(\Phi^{n}_{N})^{-3}-(\phi^{n})^{-3},\tilde{\phi}^{n+1}\rangle_{\Omega} &\le 16 \| ( \xi^n )^{-4} \|_2 \cdot \| \tilde{\phi}^n \|_4 \cdot \| \tilde{\phi}^{n+1} \|_4
   	\nonumber
	\\
& \le 16 C^*  \| \tilde{\phi}^n \|_4 \cdot \| \tilde{\phi}^{n+1} \|_4 \le 8 C^* \left( \| \tilde{\phi}^n \|_4^2 + \| \tilde{\phi}^{n+1} \|_4^2 \right) .
    \label{AG-convergence-5-3}
	\end{align}
Meanwhile, for the $\| \, \cdot \, \|_4$ norm, an application of Lemma~\ref{lem:L4 est} indicates that
	\begin{equation}
\| \tilde{\phi}^k \|_4^2 \le C \| \tilde{\phi}^k \|_{-1,h}^\frac14  \cdot \| \nabla_h \tilde{\phi}^k \|_2^\frac74 \le \tilde{C}_2 \varepsilon^{-2} \| \tilde{\phi}^k \|_{-1,h}^2 + \frac{\varepsilon^2}{16 C^*} \| \nabla_h \tilde{\phi}^k \|_2^2 ,  \quad k = n, n+1 ,
  	\label{AG-convergence-5-4}
	\end{equation}
where  Young's inequality has been applied. Then we get
	\begin{align}
   -\frac{16}{3} \langle(\Phi^{n}_{N})^{-3}-(\phi^{n})^{-3},\tilde{\phi}^{n+1}\rangle_{\Omega}
    &\le 8 C^* \tilde{C}_2 \varepsilon^{-2} \left( \| \tilde{\phi}^{n+1} \|_{-1,h}^2
    + \| \tilde{\phi}^n \|_{-1,h}^2 \right)
	\nonumber
	\\
& \quad + \frac{\varepsilon^2}{2} \left( \| \nabla_h \tilde{\phi}^{n+1} \|_2^2
     + \| \nabla_h \tilde{\phi}^n \|_2^2  \right) .
    \label{AG-convergence-5-5}
	\end{align}
The term associated with the truncation error can be controlled by the standard Cauchy inequality:
\begin{equation}
\langle\tau^{n},\tilde{\phi}^{n+1}\rangle_{-1,h}
 \le  \|\tau^{n}\|_{-1,h} \cdot \|\tilde{\phi}^{n+1}\|_{-1,h}
\le \frac12 \left( \|\tau^{n}\|_{-1,h}^{2}+\|\tilde{\phi}^{n+1}\|_{-1,h}^{2} \right) .
 \label{AG-convergence-6}
\end{equation}
Subsequently, a substitution of~\eqref{AG-convergence-2}-\eqref{AG-convergence-6}  into~\eqref{AG-convergence-1} yields
	\begin{equation}  \label{AG-convergence-7}
	\begin{split}
& \|\tilde{\phi}^{n+1}\|_{-1,h}^{2}-\|\tilde{\phi}^{n}\|_{-1,h}^{2} + \frac32 \dt \varepsilon^2 \| \nabla_{h} \tilde{\phi}^{n+1}\|_{2}^{2}
	\\
& \quad  \le  \dt \|\tau^{n}\|_{-1,h}^{2} + ( 8 C^* \tilde{C}_2 \varepsilon^{-2} + 1) \dt \left( \| \tilde{\phi}^{n+1} \|_{-1,h}^2  + \| \tilde{\phi}^n \|_{-1,h}^2 \right) + \frac{\varepsilon^2}{2} \dt \| \nabla_h \tilde{\phi}^n \|_2^2 .
	\end{split}
	\end{equation}
Finally, an application of a discrete Gronwall inequality results in the desired convergence estimate:
\begin{equation}
\begin{split}
\|\tilde{\phi}^{n+1}\|_{-1,h}
+\left( \varepsilon^2 \dt \sum_{k=0}^{n+1}\|\nabla_{h}\tilde{\phi}^{m}\|_{2}^{2}\right)^{\frac{1}{2}}
\leq C(\dt+h^{2}),
\end{split}
\end{equation}
where $C>0$ is independent of $\dt$, $h$ and $n$. This completes the proof of Theorem~\ref{thm:convergence-1st}.
\end{proof}

\begin{rem}   \label{rem: Gronwall}
It is observed that, a constant in the form of $8 C^* \tilde{C}_2 \varepsilon^{-2}$ has appeared in the coefficient of $\| \tilde{\phi}^{n+1} \|_{-1,h}^2$ on the right hand side of~\eqref{AG-convergence-7}. As a result, a constraint for the time step size, namely in the form of $C^* \tilde{C}_2 \varepsilon^{-2} \dt \le \frac{1}{32}$, is needed to obtain an inequality in the following format
	\begin{equation*}
	\begin{split}
& \|\tilde{\phi}^{n+1}\|_{-1,h}^{2} +  \varepsilon^2 \dt \sum_{k=1}^{n+1}\| \nabla_{h} \tilde{\phi}^{k}\|_{2}^{2}
\le  \dt \sum_{k=0}^n \|\tau^{n}\|_{-1,h}^{2} + ( C C^* \tilde{C}_2 \varepsilon^{-2} + 1) \dt \sum_{k=0}^n \| \tilde{\phi}^k \|_{-1,h}^2  ,
	\end{split}
	\end{equation*}
so that a discrete Gronwall inequality could be effectively applied. In fact, such a constraint could be viewed as the condition, ``provided $\dt$ and $h$ are sufficiently small", stated in Theorem~\ref{thm:convergence-1st}.
\end{rem}

	\section {A second order numerical scheme}
	\label{sec:2nd scheme}

In the first order scheme~\eqref{scheme-AG-1}, the nonlinear concave term is treated explicitly. This approach may lead to a difficulty to derive a second order accurate, energy stable scheme. Instead, we make use of an alternate convexity analysis. The following preliminary estimate is needed.
	\begin{lem}
	\label{lem:convexity-0}
For $x>0$, the function $f(x)=\frac{1}{3}x^{-8}-\frac{4}{3}x^{-2}+\frac{4}{3}A_{0}x^{2}$ is convex, provided $A_{0} \ge A_0^\star := \frac{9}{5}(\frac{2}{15})^{\frac{2}{3}}$.
	\end{lem}
	\begin{proof}
A direct calculation gives $f'' (x)=\frac{8}{3}(9x^{-10}-3x^{-4}+A_{0})$. Meanwhile, a careful application of Young inequality reveals that
	\begin{equation*}
y\leq3y^{\frac{5}{2}}+\frac{1}{3}A_{0}^\star,\quad \forall \, y>0.
	\end{equation*}
By setting $y=x^{-4},$ we conclude that $f^{''}(x)\geq0$, for any $x>0$. This, in turn, implies that $f(x)$ is convex for $x>0$.
	\end{proof}

As a result, we obtain the following alternate decomposition of $F(\phi)$:
	\[
F(\phi) =F_{c}(\phi)-F_{e}(\phi),
	\]
where
	\begin{equation*}
F_{c}(\phi)=\int_{\Omega} \left( \frac{1}{3}\phi^{-8} + \frac{4}{3}A_{0} \phi^{2} -\frac{4}{3}\phi^{-2} + \frac{\varepsilon^{2}}{2} | \nabla\phi | ^{2} \right) d \bx ,  \quad	 F_{e}(\phi)=\int_{\Omega} \, \frac{4}{3}A_{0}\phi^{2} d \bx .
	\end{equation*}

We propose the following second order scheme for the wetting and droplet coarsening equation: given $\phi^{n-1},\phi^{n}\in\mathcal{C}_{\rm per},$ find $\phi^{n+1}, \mu^{n+1}\in\mathcal{C}_{\rm per},$ such that
	\begin{align}
\frac{\frac{3}{2}\phi^{n+1}-2\phi^{n}+\frac{1}{2}\phi^{n-1}}{\dt} & = \Delta_h \mu^{n+1} ,
	\label{BDF2-AG-1}
	\\
\mu^{n+1} & =-\frac{8}{3}((\phi^{n+1})^{-9} - (\phi^{n+1})^{-3}) +\frac{8}{3}A_{0}(\phi^{n+1}-\check{\phi}^{n+1})
	\nonumber
	\\
& \quad  - A \dt \Delta_h (\phi^{n+1}-\phi^{n}) - \varepsilon^2 \Delta_h \phi^{n+1},
	\label{scheme-mu-BDF2-1}
	\end{align}
where $\check{\phi}^{n+1} :=2\phi^{n}-\phi^{n-1}$ and $A\ge 0$ is a stabilization parameter to be determined. The unique solvability and positivity-preserving properties for the second order numerical scheme are established in the following theorem. The initial data, $\phi^0,\phi^{-1}\in \mathcal{C}_{\rm per}$, will be specified later.

	\begin{thm}
	\label{CH-BDF2-positivity}
Given $\phi^{k} \in \mathcal{C}_{\rm per}$, with $\| \phi^k \|_\infty \le M_h$, $k=n, n-1$, for some $M_h >0$,  and $\overline{\phi^{n}}=\overline{\phi^{n-1}}= \beta_0$, there exists a unique solution $\phi^{n+1}\in\mathcal{C}_{\rm per}$ to \eqref{BDF2-AG-1}, with $\phi^{n+1}-\phi^n\in\mathring{\mathcal{C}}_{\rm per}$, and $\phi^{n+1} > 0$ at a point-wise level.
	\end{thm}

	\begin{proof}
We follow the notation in the proof of Theorem~\ref{CH-positivity}. Solving~\eqref{BDF2-AG-1} is equivalent to minimizing the discrete energy functional
	\begin{equation}
	\label{AG-BDF2-positive-1}
	\begin{split}
\mathcal{J}(\varphi) & =  \frac{1}{3 \dt} \nrm{ \frac{3}{2}(\varphi + \beta_0) - 2 \phi^n
 +\frac12 \phi^{n-1} }_{-1,h}^{2}
+\frac{1}{3} \langle ( \varphi+ \beta_0)^{-8} - 4 (\phi+ \beta_0)^{-2} , 1 \rangle_\Omega
	\\
& \ +4 A_0\langle (\varphi+ \beta_0)^{2}, 1 \rangle_{\Omega} + \frac{\varepsilon^{2}+A \dt}{2} \| \nabla_h \varphi \|_2^2 +\langle (\varphi+ \beta_0), A \dt \Delta_h \phi^n - \frac{8}{3} A_{0} \check{\phi}^{n+1} \rangle_{\Omega}.
	\end{split}
	\end{equation}
over the admissible set $\mathring{V}_{h}$, defined in~\eqref{def-V_h}.

To obtain the existence of a minimizer for $\mathcal{J}$ over $\mathring{V}_{h}$, we consider the closed domain $\mathring{V}_{\delta,h}$, defined in~\eqref{def-V_h-delta}. Since $\mathring{V}_{\delta,h}$ is a bounded, compact, and convex set in the subspace $\mathring{\mathcal{C}}_{\rm per}$, there exists a minimizer of $\mathcal{J}$ over $\mathring{V}_{\delta,h}$. We will show that such a minimizer could not occur on the left boundary, if $\delta$ is small enough.

To get a contradiction, assume that the minimizer of $\mathcal{J}$ occurs at the left boundary of $\mathring{V}_{\delta,h}$. Denote the minimizer by $\varphi^{\star}$, and suppose that $\varphi_{\vec{\alpha_{0}}}^{\star}=\delta- \beta_0$, where $\vec{\alpha_{0}} = (i_0,j_0,k_0)$. Then the grid function $\varphi^{\star}$ has a global minimum at $\vec{\alpha_{0}}$. Suppose that $\varphi^{\star}$ reaches its maximum value at the grid point $\vec{\alpha_{1}}=(i_1,j_1,k_1)$. Since $\overline{\varphi^{\star}}=0$, it is obvious that $M_h \geq\varphi_{\vec{\alpha_{1}}}^{\star} + \beta_0 \ge \beta_0$.

Since $\mathcal{J}$ is smooth over $\mathring{V}_{\delta,h},$ for any $\psi\in\mathring{\mathcal{C}}_{\rm per}$, the directional derivative is
	\begin{equation*}
	\begin{split}
d_{s}\mathcal{J}(\varphi^{\star}+s\psi)|_{s=0} & :=
  \frac{1}{\dt} \left\langle (-\Delta_{h})^{-1} \left(\frac{3}{2}(\varphi^{\star}+ \beta_0) - 2 \phi^{n}
   + \frac{1}{2} \phi^{n-1} \right), \psi \right\rangle_{\Omega}
	\\
&\quad +\frac{8}{3 }\langle ( -\varphi^{\star} + \beta_0)^{-9}
+ ( \varphi^{\star}+ \beta_0)^{-3} + A_{0} (\varphi^{\star}+ \beta_0 ), \psi \rangle_{\Omega}
	\\
&\quad  - (\varepsilon^{2}+ A \dt) \langle \Delta_{h} \varphi^{\star} , \psi \rangle_{\Omega} +\left\langle A \dt \Delta_{h} \phi^{n} - \frac{8}{3} A_{0} \check{\phi}^{n+1},
  \psi \right\rangle_{\Omega}.
	\end{split}
	\end{equation*}
As before, let us select the direction $\psi_{i,j,k}=\delta_{i,i_0}\delta_{j,j_0}\delta_{k,k_0}-\delta_{i,i_1}\delta_{j,j_1}\delta_{k,k_1}$. Then, the  derivative may be expressed as
	\begin{equation}
	\label{AG-BDF2-positive-4}
	\begin{split}
\frac{1}{h^{3}}d_{s} \mathcal{J}(\varphi^{\star}+ s \psi) |_{s=0}& := \frac{1}{\dt} (-\Delta_h)^{-1} \left(\frac{3}{2}(\varphi^{\star}_{\vec{\alpha_{0}}}+ \beta_0)
 - 2 \phi^{n}_{\vec{\alpha_{0}}} + \frac{1}{2}\phi^{n-1}_{\vec{\alpha_{0}}}\right)
 	\\
&\quad  -\frac{1}{\dt} (-\Delta_h)^{-1} \left(\frac{3}{2}(\varphi^{\star}_{\vec{\alpha_{1}}}+ \beta_0)
- 2 \phi^{n}_{\vec{\alpha_{1}}} + \frac{1}{2} \phi^{n-1}_{\vec{\alpha_{1}}}\right)
	\\
&\quad  +\frac{8}{3}\left( -(\varphi^{\star}_{\vec{\alpha_{0}}} + \beta_0)^{-9}
 + (\varphi^{\star}_{\vec{\alpha_{0}}} + \beta_0)^{-3}
 + A_{0} (\varphi^{\star}_{\vec{\alpha_{0}}}+ \beta_0) \right)
 	\\
&\quad  -\frac{8}{3} \left( -(\varphi^{\star}_{\vec{\alpha_{1}}}+ \beta_0)^{-9}
 + (\varphi^{\star}_{\vec{\alpha_{1}}} - \beta_0)^{-3}
 + A_{0}(\varphi^{\star}_{\vec{\alpha_{1}}}+ \beta_0) \right)
 	\\
&\quad  -(\varepsilon^{2} + A \dt) (\Delta_h \varphi^{\star}_{\vec{\alpha_{0}}}
- \Delta_h \varphi^{\star}_{\vec{\alpha_{1}}})
+ A \dt ( \Delta_h \phi^{n}_{\vec{\alpha_{0}}} - \Delta_h \phi^{n}_{\vec{\alpha_{1}}})
	\\
&\quad  -\frac{8}{3}A_{0} (\check{\phi}^{n+1}_{\vec{\alpha_{0}}}
- \check{\phi}^{n+1}_{\vec{\alpha_{1}}} ).
	\end{split}
	\end{equation}
The following estimates are derived
	\begin{equation}
	\label{AG-BDF2-positive-5-1}
	\begin{split}
-5 \tilde{C} M_h \dt^{-1} &\le \frac{1}{\dt} (-\Delta_h)^{-1}
 \left(\frac{3}{2} \phi^{\star}_{\vec{\alpha_{0}}}
-2 \phi^{n}_{\vec{\alpha_{0}}} + \frac{1}{2}\phi^{n-1}_{\vec{\alpha_{0}}} \right)
	\\
& \quad  - \frac{1}{\dt} (-\Delta_h)^{-1} \left( \frac{3}{2}\phi^{\star}_{\vec{\alpha_{1}}} - 2 \phi^{n}_{\vec{\alpha_{1}}} + \frac{1}{2} \phi^{n-1}_{\vec{\alpha_{1}}} \right) \le 5 \tilde{C} M_h \dt^{-1} ,
	\end{split}
	\end{equation}
\begin{equation}
 \Delta_h \varphi^{\star}_{\vec{\alpha_{0}}} \ge 0 , \quad
 \Delta_h \varphi^{\star}_{\vec{\alpha_{1}}} \le 0 ,  \label{AG-BDF2-positive-5-2}
\end{equation}
	\begin{equation}
\Delta_h \phi^{n}_{\vec{\alpha_{0}}} \le \frac{6M_{h}}{h^{2}} , \quad
 \Delta_h \phi^{n}_{\vec{\alpha_{1}}}<\frac{6M_{h}}{h^{2}} ,
 	\label{AG-BDF2-positive-5-3}
	\end{equation}
and
	\begin{equation}
-3M_{h} \le \check{\phi}^{n+1}_{\vec{\alpha_{0}}}
 - \check{\phi}^{n+1}_{\vec{\alpha_{1}}} \le 3M_{h} ,
 	\label{AG-BDF2-positive-5-4}
	\end{equation}
where we have made use of the fact that $ \| \phi^k \|_\infty <M_{h},\;k=n,n-1$, and where we have applied Lemma~\ref{CH-positivity-Lem-0}. Similar to the argument in~\eqref{AG-positive-10}, we obtain the following inequalities, based on the fact that $\varphi_{\vec{\alpha_{0}}}^{\star}+ \beta_0=\delta$, $\varphi_{\vec{\alpha_{1}}}^{\star} + \beta_0 \ge \beta_0$:
	\begin{align}
	 &
\frac{8}{3}(-(\varphi^{\star}_{\vec{\alpha_{0}}} + \beta_0)^{-9} + (\varphi^{\star}_{\vec{\alpha_{0}}} + \beta_0)^{-3} + A_{0}(\varphi^{\star}_{\vec{\alpha_{0}}}+ \beta_0) =\frac{8}{3}(-\delta^{-9} +\delta^{-3} + A_{0}\delta),
	\label{AG-BDF2-positive-6-1}
\\
  &
\frac{8}{3} ( -(\varphi^{\star}_{\vec{\alpha_{1}}}+ \beta_0)^{-9} + (\varphi^{\star}_{\vec{\alpha_{1}}} + \beta_0)^{-3} + A_{0} (\varphi^{\star}_{\vec{\alpha_{1}}} + \beta_0)) \ge \frac{8}{3} ( -\beta_0^{-9} + (\beta_0)^{-3} + A_{0} \beta_0).
	\label{AG-BDF2-positive-6-2}
	\end{align}
Subsequently, a substitution of~\eqref{AG-BDF2-positive-5-1} to \eqref{AG-BDF2-positive-6-2} into~\eqref{AG-BDF2-positive-4} yields the following bound:
	\begin{equation}
	\label{AG-BDF2-positive-7}
	\begin{split}
\frac{1}{h^{3}}d_{s}\mathcal{J} (\varphi^{\star}+s\psi) |_{s=0}
&\le \frac{8}{3}(-\delta^{-9} + \delta^{-3} + \beta_0^{-9} + A_{0} (\delta - \beta_0))
	\\
& \quad +8A_{0}M_{h} + 12M_{h} A \dt h^{-2} + 10 \tilde{C} M_h \dt^{-1}.
	\end{split}
	\end{equation}
We observe that $(\dt)^{-1}$ and $\dt h^{-2}$ are constants, since $\dt$ and $h$ are assumed fixed, though the terms become singular as $\dt, h \rightarrow0$. The rest of the analysis follows the same arguments as in the proof of Theorem~\ref{CH-positivity}; the details are left to interested readers. This finishes the proof of Theorem~\ref{CH-BDF2-positivity}.
	\end{proof}

In the following Theorem, we could prove that a modified energy stability is available for the second order $BDF$ scheme~\eqref{BDF2-AG-1}, provided that $A\geq\frac{4}{9}A_{0}^{2}$.
	\begin{thm}
	\label{CH-BDF2-energy stability}
With the same assumptions as in Theorem~\ref{CH-BDF2-positivity}, we have the following stability for the proposed numerical scheme~\eqref{BDF2-AG-1}: provided $A\geq\frac{4}{9}A_{0}^{2}$,
	\begin{equation}
\tilde{F}_{h} (\phi^{n+1},\phi^{n}) \le \tilde{F}_{h}(\phi^{n},\phi^{n-1}),
  	\label{AG-BDF2-stability-0}
	\end{equation}
where
	\begin{equation}
\tilde{F}_h (\phi^{n+1}, \phi^n) := F_h (\phi^{n+1}) + \frac{1}{4 \dt} \| \phi^{n+1} - \phi^n \|_{-1,h}^2 + \frac43 A_0 \| \phi^{n+1} - \phi^n \|_2^2 ,
	\label{mod energy-BDF2-1}
	\end{equation}
and
	\[
F_{h}(\phi) := \frac{1}{3} \langle \phi^{-8} , 1 \rangle_\Omega + \frac{\varepsilon^{2}}{2} \| \nabla_h \phi \|_2^2 - \frac{4}{3} \langle  \phi^{-2} , 1 \rangle_\Omega ,
	\]
for any $\dt$, $h>0$.
	\end{thm}
	\begin{proof}
By taking the inner products of the terms of~\eqref{BDF2-AG-1} with $(-\Delta_h)^{-1}(\phi^{n+1}-\phi^{n})$, we can derive the following inequalities: for the time difference,
	\begin{equation}
	\label{AG-BDF2-stability-1}
	\begin{split}
& \frac{1}{\dt} \left\langle \frac{3}{2}\phi^{n+1 }- 2\phi^{n} + \frac{1}{2} \phi^{n-1} , (-\Delta_h)^{-1} \left(\phi^{n+1}-\phi^{n}\right) \right\rangle_{\Omega}
 	\\
& \quad =  \frac{3}{2 \dt} \|\phi^{n+1}-\phi^{n}\|_{-1,h}^{2}
- \frac{1}{2 \dt} \langle \phi^{n+1} - \phi^{n},\phi^{n} - \phi^{n-1} \rangle_{-1,h}
	\\
& \quad \ge  \frac{1}{\dt} \left( \frac{5}{4} \|\phi^{n+1}-\phi^{n}\|_{-1,h}^{2}
 - \frac{1}{4}\|\phi^{n}-\phi^{n-1}\|_{-1,h}^{2} \right);
	\end{split}
	\end{equation}
for the singular terms,
	\begin{equation}
	\label{AG-BDF2-stability-2-1}
	\begin{split}
& \left\langle - \Delta_h \left(-\frac{8}{3} ( (\phi^{n+1})^{-9} - (\phi^{n+1})^{-3})
+ \frac{8}{3}A_{0} \phi^{n+1} \right), (-\Delta_h)^{-1} \left(\phi^{n+1}-\phi^{n}\right) \right\rangle_{\Omega}
	\\
& \quad =  \left\langle-\frac{8}{3}\left((\phi^{n+1})^{-9} + (\phi^{n+1})^{-3}\right)
+ \frac{8}{3} A_{0} \phi^{n+1}, \phi^{n+1}-\phi^{n} \right\rangle_{\Omega}
	\\
& \quad \ge  \frac{1}{3} \left( \left\langle(\phi^{n+1})^{-8}-4(\phi^{n+1})^{-2}+4A_{0}(\phi^{n+1})^{2},1\rangle_{\Omega} - \langle(\phi^{n})^{-8}-4(\phi^{n})^{-2}+4A_{0}(\phi^{n})^{2},1\right\rangle_{\Omega} \right) ;
	\end{split}
	\end{equation}
for the highest-order diffusion term,
	\begin{equation}
	\label{AG-BDF2-stability-3}
	\begin{split}
 \left\langle \Delta_h^{2}\phi^{n+1}, (-\Delta_h)^{-1} (\phi^{n+1}-\phi^{n}) \right\rangle_{\Omega} &   = \left\langle \nabla_h \phi^{n+1}, \nabla_h (\phi^{n+1}-\phi^{n}) \right\rangle_{\Omega}
	\\
&  =\frac{1}{2} \left( \|\nabla_h \phi^{n+1}\|_2^2 - \| \nabla_h \phi^n \|_2^2 + \| \nabla_h (\phi^{n+1}-\phi^{n}) \|_2^2 \right);
	\end{split}
	\end{equation}
for stabilization terms,
	\begin{equation}
A \dt \left\langle \Delta_h^{2} \left(\phi^{n+1}-\phi^{n}\right),( -\Delta_h)^{-1} (\phi^{n+1}-\phi^{n}) \right\rangle_{\Omega} =A \dt \| \nabla_h (\phi^{n+1}-\phi^{n}) \|_2^2 ;
	\label{AG-BDF2-stability-4}
	\end{equation}
and for the extrapolation terms,
	\begin{equation}
	\label{AG-BDF2-stability-5}
	\begin{split}
& \left\langle \Delta_h (2\phi^{n}-\phi^{n-1}), (-\Delta_h)^{-1} (\phi^{n+1}-\phi^{n}) \right\rangle_{\Omega}
	\\
& \quad = - \left\langle 2\phi^{n} - \phi^{n-1}, \phi^{n+1} - \phi^{n} \right\rangle_{\Omega} \ge -\frac{1}{2} (\|\phi^{n+1}\|_{2}^{2} - \|\phi^{n}\|_{2}^{2}) - \frac{1}{2} \|\phi^{n}-\phi^{n-1}\|_2^{2}.
	\end{split}
	\end{equation}
Estimate~\eqref{AG-BDF2-stability-2-1} is based on the convexity of
	\[
\frac{1}{3}\phi^{-8}-\frac{4}{3}\phi^{-2}+\frac{4}{3}A_{0}\phi^{2},
	\]
as guaranteed by Lemma~\ref{lem:convexity-0}. Meanwhile, an application of Cauchy's inequality yields
	\begin{equation}
\frac{1}{\dt} \| \phi^{n+1} - \phi^n \|_{-1,h}^2 + A \dt \| \nabla_h ( \phi^{n+1} - \phi^n ) \|_2^2 \ge 2 A^{\frac12}  \| \phi^{n+1} - \phi^n \|_2^2 .
	\label{CH-BDF2-stability-6}
	\end{equation}
A combination of \eqref{AG-BDF2-stability-1}-\eqref{CH-BDF2-stability-6} yields
	\begin{equation}
	\begin{split}
 & F_h (\phi^{n+1})  - F_h (\phi^n) + \frac{1}{4 \dt}  \left(  \| \phi^{n+1} - \phi^n \|_{-1, h}^2 - \| \phi^n - \phi^{n-1} \|_{-1, h}^2 \right)
	\\
&  + \frac{4}{3}A_{0} \left(\|\phi^{n+1}-\phi^{n}\|_{2}^{2} - \|\phi^{n}-\phi^{n-1}\|_{2}^{2} \right)
	\\
& \hspace{1in}  \leq \left(-2A^{\frac12}+\frac{4}{3}A_{0}\right) \|\phi^{n+1}-\phi^{n} \|_{2}^{2}
	\\
& \hspace{1in}  \le 0,
	\end{split}
	\end{equation}
provided that $A\geq\frac{4}{9}A_{0}^{2}$. Therefore,  we get the energy estimate~\eqref{AG-BDF2-stability-0}. This completes the proof of Theorem~\ref{CH-BDF2-energy stability}.
	\end{proof}

	\begin{rem}
There have been a few recent works describing BDF-type schemes for certain gradient flow models -- such as the Cahn-Hilliard equation~\cite{cheng2019a, yan18}, the epitaxial thin film equation~\cite{fengW18b, Hao2020, LiW18, Meng2020}, and the square phase field crystal model~\cite{cheng2019d} -- in which an energy stability was theoretically established. In the present case, a Douglas-Dupont-type regularization has to be included in the BDF2 scheme, and a careful analysis reveals an energy stability with respect to a modified energy. Meanwhile, the present analysis for a gradient flow with a strongly singular energy potential is non-trivial; but see the related analyses in~\cite{chen19b, dong20a}.
	\end{rem}
	
\begin{rem}
In the proposed second order numerical scheme~\eqref{BDF2-AG-1}-\eqref{scheme-mu-BDF2-1}, an additional quadratic term $\frac43 A_0 \phi^2$ makes the nonlinear singular energy functional convex, so that the energy stability analysis could be derived. Meanwhile, there are some alternate ways to get a convex-concave decomposition of the physical energy, such as taking a second order approximation to the concave term $\phi^{-3}$ at $t^{n+1}$ using an explicit extrapolation formula: $2 ( \phi^n )^{-3} -  ( \phi^{n-1} )^{-3}$. In fact, such an approach has been widely applied to the standard Cahn-Hilliard equation, in which the concave term is linear; see the related works~\cite{cheng2019a, yan18}. However, for the droplet liquid film model~\eqref{equation-AG}, this approach will lead to a theoretical difficulty to justify the energy stability of the corresponding numerical scheme, due to the singular and nonlinear nature of the concave term $\phi^{-3}$. Instead, we combine the singular convex term and singular concave term, as well as an additional quadratic term, so that their combined energy becomes convex, and this approach leads to perfect theoretical properties. Also see a few related works~\cite{LiD2017, LiD2017b} for the analysis of the stabilization numerical methods applied to the standard Cahn-Hilliard equation.
\end{rem}

\begin{rem}
The proposed scheme~\eqref{BDF2-AG-1} is a two-step method. In addition to the initial data $\phi^0$, as given by~\eqref{initial data-1}, we need a ``ghost" point data, $\phi^{-1}$, a numerical approximation to the phase variable at $t^{-1} = - \dt$. Due to the long-stencil nature of the BDF2 temporal discretization, we take $\phi^{-1} = \phi^0 - \dt \Delta_h \mu_h^0$, in which $\mu_h$ stands for the finite difference approximation to the chemical potential; also see a similar treatment for the Cahn-Hilliard model in~\cite{guo16}. Such a ``ghost" point initialization gives a local consistency estimate $\phi^{-1} - \Phi^{-1} = O (\dt^2 + h^2)$, while its substitution into~\eqref{BDF2-AG-1} implies a first order truncation error, at $n=0$. On the other hand, because of the temporal derivative involved, such a fact does not spoil the truncation error for $n \ge 1$.

This approach does not affect the positivity-preserving property of the numerical solution for $\phi^{n+1}$, $n \ge 0$, since $\phi^0$ and $\phi^{-1}$ always stay bounded within one time step. Such an initialization method does not affect the modified energy stability estimate~\eqref{AG-BDF2-stability-0}, either. Meanwhile, to derive a uniform bound for the original energy functional of the numerical solution, we have to add two $O (\dt^2)$ correction terms:
	\begin{equation*}
\begin{aligned}
   F_h (\phi^{n+1}) \le & \tilde{F}_{h} (\phi^{n+1},\phi^{n}) \le \tilde{F}_{h}(\phi^{n},\phi^{n-1})
   ...
\\
   \le & \tilde{F}_{h}(\phi^{0},\phi^{-1})
   = F_h (\phi^0) + \frac{1}{4 \dt} \| \phi^0 - \phi^{-1} \|_{-1,h}^2
     + \frac43 A_0 \| \phi^0 - \phi^{-1} \|_2^2  .
\end{aligned}
	\end{equation*}
\end{rem}

The convergence result for the second order scheme is stated in the following theorem.


\begin{thm}  \label{thm:BDF2-convergence}
Given initial data $\Phi( \, \cdot \, ,t=0)\in C^{6}_{\rm per}(\Omega)$, suppose the exact solution for~\eqref{equation-AG} is of regularity class
	\[
\mathcal{R}_2:=H^{3}(0,T;C_{\rm per}(\Omega)) \cap H^2 (0,T;C^{2}_{\rm per}(\Omega)) \cap L^{\infty}(0,T;C^{6}_{\rm per}(\Omega)).
	\]
Then, provided $\dt$ and $h$ are sufficiently small, for all positive integers $n$, such that $t_{n}\le T,$ we have the following convergence estimate for the numerical solution~\eqref{BDF2-AG-1}
\begin{equation}
\| \tilde{\phi}^n \|_{-1,h} +  \Bigl( \varepsilon^2 \dt   \sum_{m=1}^{n} \| \nabla_h \tilde{\phi}^m \|_2^2 \Bigr)^\frac12  \le C ( \dt^2 + h^2 ),
	\label{AG-BDF2-convergence-0}
\end{equation}
where $C>0$ is independent of $n$,$\dt,$ and $h$.
\end{thm}


	\begin{proof}
A careful consistency analysis indicates the following truncation error estimate:
\begin{equation}  \label{AG-BDF2-consistency-1}
\begin{split}
   \frac{\frac32 \Phi_N^{n+1} - 2 \Phi_N^n + \frac12 \Phi_N^{n-1}}{\dt}  =&
    \Delta_h
 \Bigl( - \frac83 ( ( \Phi_N^{n+1} )^{-9} - ( \Phi_N^{n+1} )^{-3} ) + \frac83 A_0
( \Phi_N^{n+1} - \check{\Phi}_N^{n+1} )
\\
  &
  - \varepsilon^2 \Delta_h \Phi_N^{n+1}
  - A \dt \Delta_h ( \Phi_N^{n+1} - \Phi_N^n ) \Bigr) + \tau^n ,
\end{split}
\end{equation}
with $\check{\Phi}_N^n = 2 \Phi_N^n - \Phi_N^{n-1}$, $\| \tau^n \|_{-1,h} \le C (\dt^2 + h^2)$.
In turn, subtracting the numerical scheme~\eqref{BDF2-AG-1} from~\eqref{AG-BDF2-consistency-1} gives
	\begin{equation}
	\label{AG-BDF2-consistency-2}
	\begin{split}
\frac{\frac32 \tilde{\phi}^{n+1} - 2 \tilde{\phi}^n + \frac12 \tilde{\phi}^{n-1}}{\dt} &  =
   \Delta_h
 \Bigl( - \frac83 ( ( \Phi_N^{n+1} )^{-9} - ( \Phi_N^{n+1} )^{-3} - A_0 \Phi_N^{n+1} )
	\\
& \quad + \frac83 ( ( \phi^{n+1} )^{-9} - ( \phi^{n+1} )^{-3} - A_0 \phi^{n+1} )
  - \frac83 A_0 \tilde{\check{\phi}}^{n+1}
	\\
& \quad  - \varepsilon^2 \Delta_h \tilde{\phi}^{n+1}
   - A \dt \Delta_h ( \tilde{\phi}^{n+1} - \tilde{\phi}^n ) \Bigr)  + \tau^n ,
\end{split}
\end{equation}
with $\tilde{\check{\phi}}^{n+1} = 2 \tilde{\phi}^n - \tilde{\phi}^{n-1}$.
Taking the discrete inner product of~\eqref{AG-BDF2-consistency-2} with $2 (-\Delta_h)^{-1} \tilde{\phi}^{n+1}$ yields
	\begin{align}
& \left\langle 3 \tilde{\phi}^{n+1} - 4 \tilde{\phi}^n + \tilde{\phi}^{n-1} ,
  \tilde{\phi}^{n+1}  \right\rangle_{-1, h}
  - 2 \varepsilon^2 \dt \ciptwo{ \tilde{\phi}^{n+1} }{ \Delta_h \tilde{\phi}^{n+1} } -  2 A \dt \langle \Delta_h ( \tilde{\phi}^{n+1} - \tilde{\phi}^n ) , \tilde{\phi}^{n+1} \rangle_\Omega
	\nonumber
	\\
& \quad +  \frac{16}{3} \dt \ciptwo{ - ( ( \Phi_N^{n+1} )^{-9} - ( \Phi_N^{n+1} )^{-3} - A_0 \Phi_N^{n+1} )  + ( ( \phi^{n+1} )^{-9} - ( \phi^{n+1} )^{-3} - A_0 \phi^{n+1} )  }{ \tilde{\phi}^{n+1} }
	\nonumber
	\\
& \qquad =  \frac{16}{3} A_0 \dt \ciptwo{ \tilde{\check{\phi}}^{n+1} }{ \tilde{\phi}^{n+1} }  + 2 \dt \langle \tau^n , \tilde{\phi}^{n+1} \rangle_{-1,h}.
	\label{AG-BDF2-convergence-1}
	\end{align}
For the temporal derivative stencil, the following identity is valid:
      \begin{equation}  \label{AG-BDF2-convergence-2}
 \begin{split}
      \left\langle 3 \tilde{\phi}^{n+1} - 4 \tilde{\phi}^n + \tilde{\phi}^{n-1} ,
  \tilde{\phi}^{n+1}  \right\rangle_{-1, h}
	=& \frac12 \Bigl( \| \tilde{\phi}^{n+1} \|_{-1, h}^2 - \| \tilde{\phi}^n \|_{-1, h}^2
\\
  &
	+ \| 2 \tilde{\phi}^{n+1} - \tilde{\phi}^n \|_{-1, h}^2
	- \| 2 \tilde{\phi}^n - \tilde{\phi}^{n-1} \|_{-1, h}^2
\\
  &
  + \| \tilde{\phi}^{n+1}  - 2 \tilde{\phi}^n + \tilde{\phi}^{n-1} \|_{-1, h}^2 \Bigr) .
\end{split}
	\end{equation}
The estimate for the surface diffusion term is straightforward:
\begin{eqnarray}
  - \langle \tilde{\phi}^{n+1} , \Delta_h \tilde{\phi}^{n+1} \rangle_\Omega  = \| \nabla_h \tilde{\phi}^{n+1} \|_2^2 .  \label{AG-BDF2-convergence-3}
\end{eqnarray}
For the nonlinear inner product, the convex nature of $\frac18 \phi^{-8} - \frac12 \phi^{-2} + \frac12 A_0 \phi^2$ (at a point-wise level) yields the following result:
	\begin{eqnarray}
	\ciptwo{ - ( ( \Phi_N^{n+1} )^{-9} - ( \Phi_N^{n+1} )^{-3} - A_0 \Phi_N^{n+1} )  + ( ( \phi^{n+1} )^{-9} - ( \phi^{n+1} )^{-3} - A_0 \phi^{n+1} )  }{ \tilde{\phi}^{n+1} }
	\ge 0 . \label{AG-BDF2-convergence-4}
	\end{eqnarray}
The term associated with the truncation error can be controlled by the standard Cauchy inequality:
	\begin{equation}
2 \langle \tau^n , \tilde{\phi}^{n+1} \rangle_{-1,h} \le  2 \| \tau^n \|_{-1,h} \cdot  \| \tilde{\phi}^{n+1} \|_{-1,h}  \le  \| \tau^n \|_{-1,h}^2 +  \| \tilde{\phi}^{n+1} \|_{-1,h}^2 .
	\label{AG-BDF2-convergence-5}
	\end{equation}
For the concave expansive error term, a similar inequality is available:
	\begin{eqnarray}
 \frac{16}{3} A_0 \ciptwo{ \tilde{\check{\phi}}^{n+1} }{ \tilde{\phi}^{n+1} }
 &\le& \frac{16}{3} A_0 \| \tilde{\check{\phi}}^{n+1} \|_{-1,h}  \| \nabla_h \tilde{\phi}^{n+1} \|_2
  \le \frac{64}{9} A_0^2 \varepsilon^{-2} \| \tilde{\check{\phi}}^{n+1} \|_{-1,h}^2  + \varepsilon^2 \| \nabla_h \tilde{\phi}^{n+1} \|_2
	\nonumber
	\\
& \le & \frac{64}{9} A_0^2 \varepsilon^{-2} ( 8 \| \tilde{\phi}^n \|_{-1,h}^2
 + 2  \| \tilde{\phi}^{n-1} \|_{-1,h}^2) + \varepsilon^2 \| \nabla_h \tilde{\phi}^{n+1} \|_2 .
	\label{AG-BDF2-convergence-6}
	\end{eqnarray}
In addition, the following identity could be derived for the artificial diffusion term:
\begin{equation}   \label{AG-BDF2-convergence-7}
\begin{split}
  - 2 \langle \Delta_h ( \tilde{\phi}^{n+1} - \tilde{\phi}^n ) , \tilde{\phi}^{n+1} \rangle_\Omega
  = & 2 \langle \nabla_h ( \tilde{\phi}^{n+1} - \tilde{\phi}^n ) , \nabla_h \tilde{\phi}^{n+1} \rangle_\Omega
\\
  =&
  \| \nabla_h \tilde{\phi}^{n+1} \|_2^2 - \| \nabla_h \tilde{\phi}^n \|_2^2
  +  \| \nabla_h ( \tilde{\phi}^{n+1} - \tilde{\phi}^n ) \|_2^2 .
\end{split}
	\end{equation}  	
Subsequently, a substitution of \eqref{AG-BDF2-convergence-2} -- \eqref{AG-BDF2-convergence-7} into~\eqref{AG-BDF2-convergence-1} yields
	\begin{equation}  \label{AG-BDF2-convergence-8}
\begin{split}
	&
\| \tilde{\phi}^{n+1} \|_{-1,h}^2 - \| \tilde{\phi}^n \|_{-1,h}^2
+ \| 2 \tilde{\phi}^{n+1} - \tilde{\phi}^n \|_{-1, h}^2
	- \| 2 \tilde{\phi}^n - \tilde{\phi}^{n-1} \|_{-1, h}^2
\\
&
+ A \dt^2 ( \| \nabla_h \tilde{\phi}^{n+1} \|_2^2 - \| \nabla_h \tilde{\phi}^n \|_2^2 )
+ \varepsilon^2 \dt \| \nabla_h \tilde{\phi}^{n+1} \|_2^2
	\\
 \le& \frac{128}{9} A_0^2 \varepsilon^{-2} \dt ( 4 \| \tilde{\phi}^n \|_{-1,h}^2
 +  \| \tilde{\phi}^{n-1} \|_{-1,h}^2) +  2 \dt \| \tau^n \|_{-1,h}^2  .
\end{split}
	\end{equation}
Similar to the arguments in Remark~\ref{rem: Gronwall}, under a constraint for the time step size, namely in the form of $\frac{128}{9} A_0^2 \varepsilon^{-2} \dt \le \frac{1}{32}$, we are able to apply the  discrete Gronwall inequality and obtain the desired convergence estimate:
	\begin{equation}
\| \tilde{\phi}^{n+1} \|_{-1,h} + \Bigl( \varepsilon^2 \dt \sum_{k=0}^{n+1} \| \nabla_h \tilde{\phi}^m \|_2^2 \Bigr)^{1/2}  \le C ( \dt^2 + h^2) ,
	\label{AG-BDF2-convergence-9}
	\end{equation}
where  $C>0$ is independent of $\dt$, $h$, and $n$. This completes the proof of  Theorem~\ref{thm:BDF2-convergence}.
	\end{proof}

	\section {The numerical solver and the numerical results}
	\label{sec:numerical results}
	
A preconditioned steepest descent (PSD) iteration algorithm is used to implement the proposed numerical schemes~\eqref{scheme-AG-1} and \eqref{BDF2-AG-1}, following the practical and  theoretical framework in~\cite{feng2017}. We give the details for the first order scheme~\eqref{scheme-AG-1}; the details for the second order one~\eqref{BDF2-AG-1} will be similar. It is clear that~\eqref{scheme-AG-1} can be recast as a minimization of the discrete convex energy functional~\eqref{AG-positive-1}, which becomes the discrete variation of \eqref{AG-positive-1} set equal to zero: $\mathcal{N}_h [\phi] = f^n$, where
	\[
 \mathcal{N}_h [\phi] := \frac{1}{\dt} (-\Delta_h)^{-1} ( \phi - \phi^n )   - \frac{8}{3} \phi^{-9}  - \varepsilon^2 \Delta_h \phi \quad \mbox{and} \quad f^n := - \frac{8}{3}(\phi^{n})^{-3}  .
	\]
The essential idea of the PSD solver is to use a linearized version of the nonlinear operator as a pre-conditioner. Specifically, the preconditioner, $\mathcal{L}_h: \mathring{\mathcal{C}}_{\rm per} \to \mathring{\mathcal{C}}_{\rm per}$,  is defined as 		\[
{\mathcal L}_h [\psi] :=   \frac{1}{\dt} (-\Delta_h)^{-1} \psi   + \psi - \varepsilon^2 \Delta_h \psi ,
	\]
and is a positive, symmetric operator. Specifically, this ``metric" is used to find an appropriate search direction for the steepest descent solver~\cite{feng2017}. Given the current iterate $\phi_n\in \mathcal{C}_{\rm per}$, we define the following \emph{search direction} problem: find $d_n \in \mathcal{C}_{\rm per}$ such that
\[
{\mathcal L}_h [d_n] = r_n-\overline{r_n} , \quad r_n := f^n -\mathcal{N}_h [\phi_n],
\]
where $r_n$ is the nonlinear residual of the $n^{\rm th}$ iterate $\phi_n$. Of course, this equation can be efficiently solved using the Fast Fourier Transform (FFT). Subsequently, the next iterate is obtained as
	\begin{equation}
\phi_{n+1} := \phi_n + \alpha_n d_n,
	\end{equation}
where $\alpha_n\in\mathbb{R}$ is the unique solution to the steepest descent line minimization problem
	\begin{equation}
\alpha_n := \operatorname*{argmin}_{\alpha\in\mathbb{R}} {\cal J} [\phi_n + \alpha d_n]= \operatorname*{argzero}_{\alpha\in\mathbb{R}}\ciptwo{\mathcal{N}_h[\phi_n+\alpha d_n]-f^n}{d_n} .
	\label{eqn-search}
	\end{equation}
Following similar techniques reported in~\cite{feng2017}, a theoretical analysis ensures a geometric convergence of the iteration sequence. Also see~\cite{chen20c, cheng2019d, fengW18a, fengW18b} for the applications of the PSD solver to various gradient flow models.

	\begin{rem}
We observe that the PSD method can be viewed as a quasi-Newton method, with an orthogonalization (line search) step. Indeed, $\mathcal{L}_h$ may be viewed as an approximation of the Jacobian. To fit more neatly into the framework of a traditional quasi-Newton method, one could just take step size equal to 1, so that the correction is just $\phi_{n+1} := \phi_n +  d_n$. Alternatively, one can just use quadratic line search methods to obtain an approximation of $\alpha_n$, call it $\alpha_n^{\rm q}$, to obtain a sufficiently good approximation that one can still prove a geometric convergence rate that is independent of $N$.
	\end{rem}

\subsection{Convergence test for the numerical schemes}

In this subsection we perform a numerical accuracy check for the proposed numerical schemes~\eqref{scheme-AG-1} and \eqref{BDF2-AG-1}. The computational domain is chosen as $\Omega = (0,1)^2$, and the exact profile for the phase variable is set to be
	\begin{equation}
\Phi (x,y,t) = 1 + \frac{1}{2 \pi} \sin( 2 \pi x) \cos(2 \pi y) \cos( t) .
	\label{AC-1}
	\end{equation}
In particular, we see that the exact profile is positive at a point-wise level, so that the computation will not cause any singularity issue. To make $\Phi$ satisfy the original PDE \eqref{equation-AG}, we have to add an artificial, time-dependent forcing term. Then the proposed scheme, either the first order one~\eqref{scheme-AG-1}, or the second order algorithm \eqref{BDF2-AG-1}, can be implemented to solve for (\ref{equation-AG}).

In the accuracy check for the first order scheme~\eqref{scheme-AG-1}, we fix the spatial resolution as $N=256$ (with $h=\frac{1}{256}$), so that the spatial numerical error is negligible. The final time is set as $T=1$, and the surface diffusion parameter is taken as $\varepsilon=0.5$.
Naturally, a sequence of time step sizes are taken as $\dt=\frac T{N_T}$, with $N_T = 100:100:1000$.  The expected temporal numerical accuracy assumption $e=C \dt$ indicates that $\ln |e|=\ln (CT) - \ln N_T$, so that we plot $\ln |e|$ vs. $\ln N_T$ to demonstrate the temporal convergence order. The fitted line displayed in Figure~\ref{fig1} shows an approximate slope of -0.9704, which in turn verifies a nice first order temporal convergence order, in both the discrete $\ell^2$ and $\ell^\infty$ norms. 


	\begin{figure}
	\begin{center}
\includegraphics[width=3.0in]{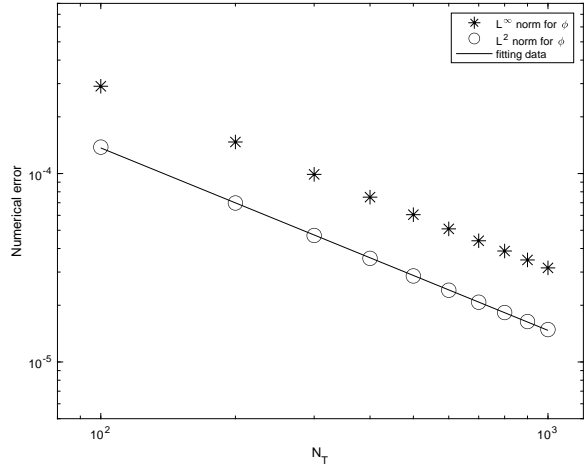}
	\end{center}
\caption{The discrete $\ell^2$ and $\ell^\infty$ numerical errors vs. temporal resolution  $N_T$ for $N_T = 100:100:1000$, with a spatial resolution $N=256$. The numerical results are obtained by the computation using the first order scheme~\eqref{scheme-AG-1}. The surface diffusion parameter is taken to be $\varepsilon=0.5$. The data lie roughly on curves $CN_T^{-1}$, for appropriate choices of $C$, confirming the full first-order accuracy of the scheme.}
	\label{fig1}
	\end{figure}
	
In the accuracy test for the second order scheme~\eqref{BDF2-AG-1}, we set the time size as $\dt = \frac12 h$, with $h = \frac{1}{N}$, so that the second order accuracy in both time and space could be confirmed. Again, the final time set as by $T=1$, and the surface diffusion parameter is given by $\varepsilon=0.5$. The artificial regularization parameters are taken as $A_{0} = A_0^\star = \frac{9}{5}(\frac{2}{15})^{\frac{2}{3}}$, $A = \frac49 A_0^2$, so that the convexity condition and stability requirement are satisfied. A sequence of spatial resolutions are taken as $N = 48:16:192$.  The expected temporal numerical accuracy assumption $e=C ( \dt^2 + h^2)$ indicates that $\ln |e|=\ln C - 2 \ln N$, so that we plot $\ln |e|$ vs. $\ln N$ to demonstrate the temporal convergence order. The fitted line displayed in Figure~\ref{fig2} shows an approximate slope of -1.9974, which in turn verifies a perfect second convergence order, in both time and space. 


	\begin{figure}
	\begin{center}
\includegraphics[width=3.0in]{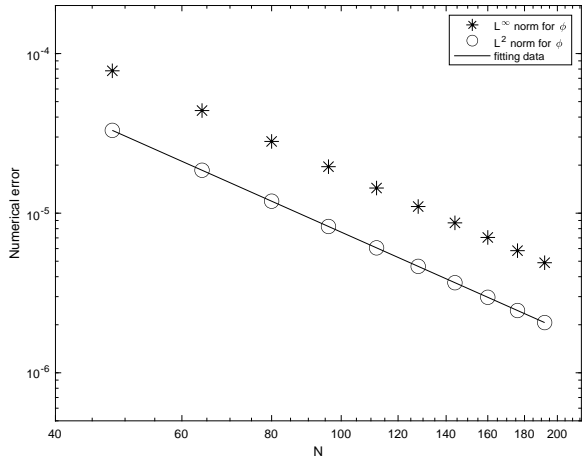}
	\end{center}
\caption{The discrete $\ell^2$ and $\ell^\infty$ numerical errors vs. spatial resolution  $N$ for $N = 48:16:192$, and the time step size is set as $\dt = \frac12 h$. The numerical results are obtained by the computation using the second order scheme~\eqref{BDF2-AG-1}. The surface diffusion parameter is taken to be $\varepsilon=0.5$. The data lie roughly on curves $CN^{-2}$, for appropriate choices of $C$, confirming the full second-order accuracy of the scheme, in both time and space.}
	\label{fig2}
	\end{figure}

\subsection{Numerical simulation of coarsening process} 

The nonlinear term in \eqref{equation-AG} associated with the  Lennard-Jones-type potential gives preference to rotationally invariant patterns. In this subsection, we perform a two-dimensional numerical simulation showing the coarsening process. The computational domain is set as $\Omega =(0,L)^2$, with $L=12.8$, and the interface width parameter is taken as $\varepsilon=0.02$. The initial data are given by
\begin{equation}
	\label{initial data-2}
 \phi^0_{i,j}= 2 + 0.1\cdot (2r_{i,j}-1), \quad
 \mbox{$r_{i,j}$ are uniformly distributed random numbers in $[0, 1]$} .
	\end{equation}
The second order numerical scheme~\eqref{BDF2-AG-1}	is implemented for this simulation. For the temporal step size $\dt$, we use increasing values of $\dt$ in the time evolution: $\dt = 0.001$ on the time interval $[0,100]$, $\dt = 0.004$ on the time interval $[100, 500]$, $\dt = 0.008$ on the time interval $[500, 2000]$ and $\dt = 0.02$ on the time interval $[2000, 6000]$. Whenever a new time step size is applied, we initiate the two-step numerical scheme by  taking $\phi^{-1} = \phi^0$, with the initial data $\phi^0$ given by the final time output of the last time period. The time snapshots of the evolution by using $\varepsilon=0.02$ are presented in Figure~\ref{fig3}, with significant coarsening observed in the system.  At early times many small hills (yellow) are present, with flat base (blue).  At the final time, $t= 6000$, a single hill structure emerges, and further coarsening is not possible. 

\begin{figure}[h]
	\begin{center}
		\begin{subfigure}{0.48\textwidth}
			\includegraphics[height=0.48\textwidth,width=0.48\textwidth]{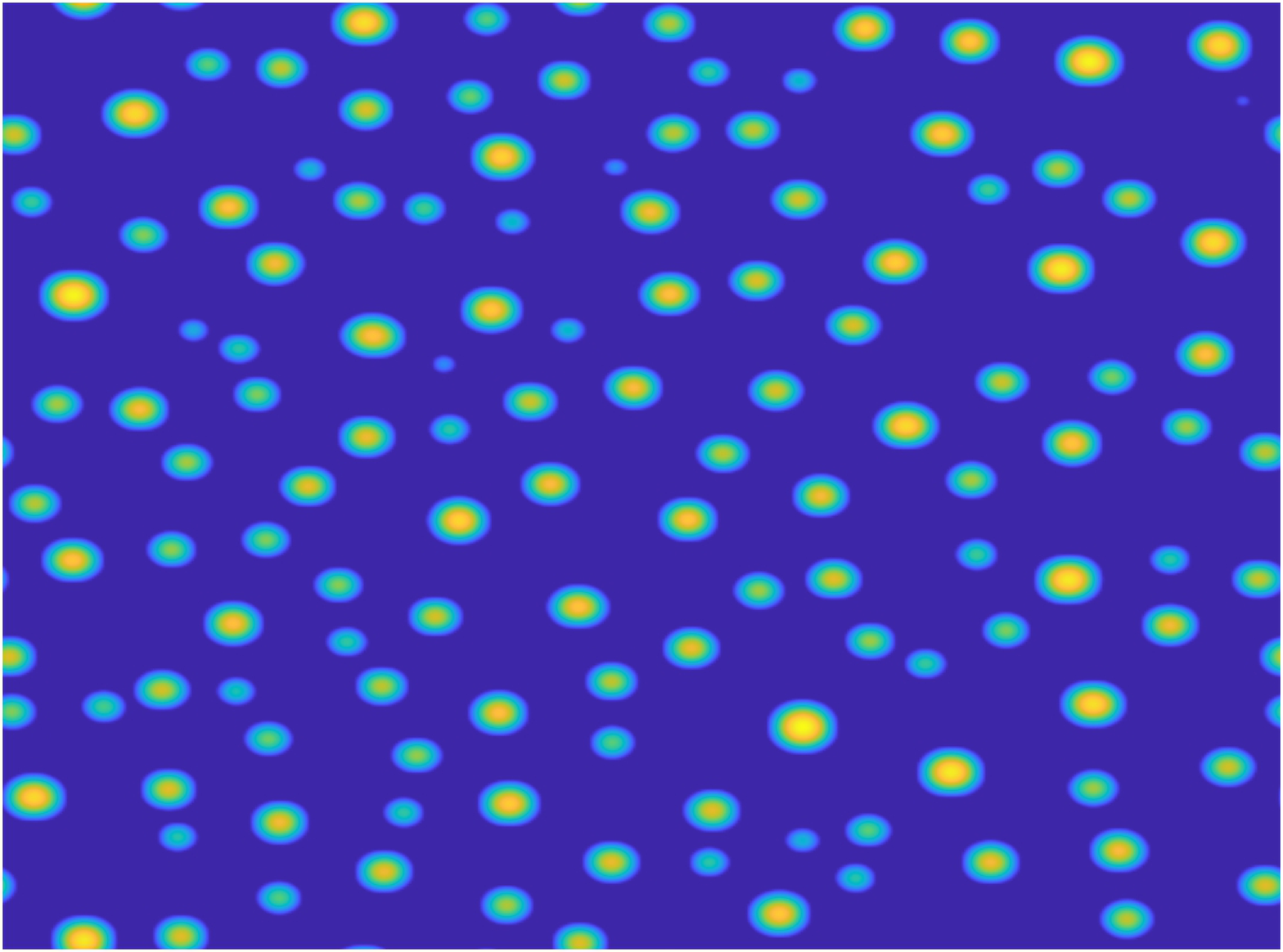}
			\includegraphics[height=0.48\textwidth,width=0.48\textwidth]{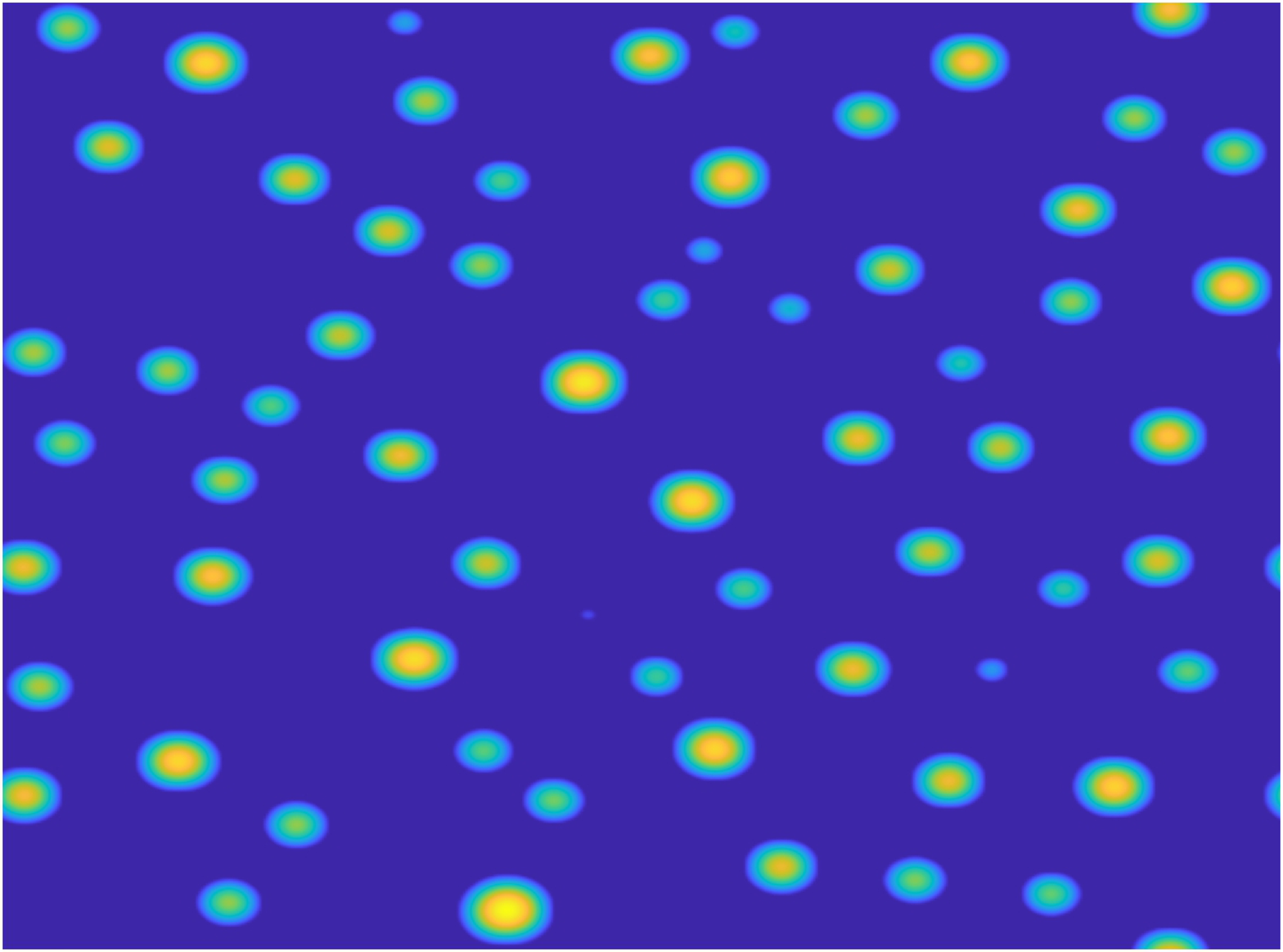}
			\caption*{$t=6, 20$}
		\end{subfigure}
		\begin{subfigure}{0.48\textwidth}
			\includegraphics[height=0.48\textwidth,width=0.48\textwidth]{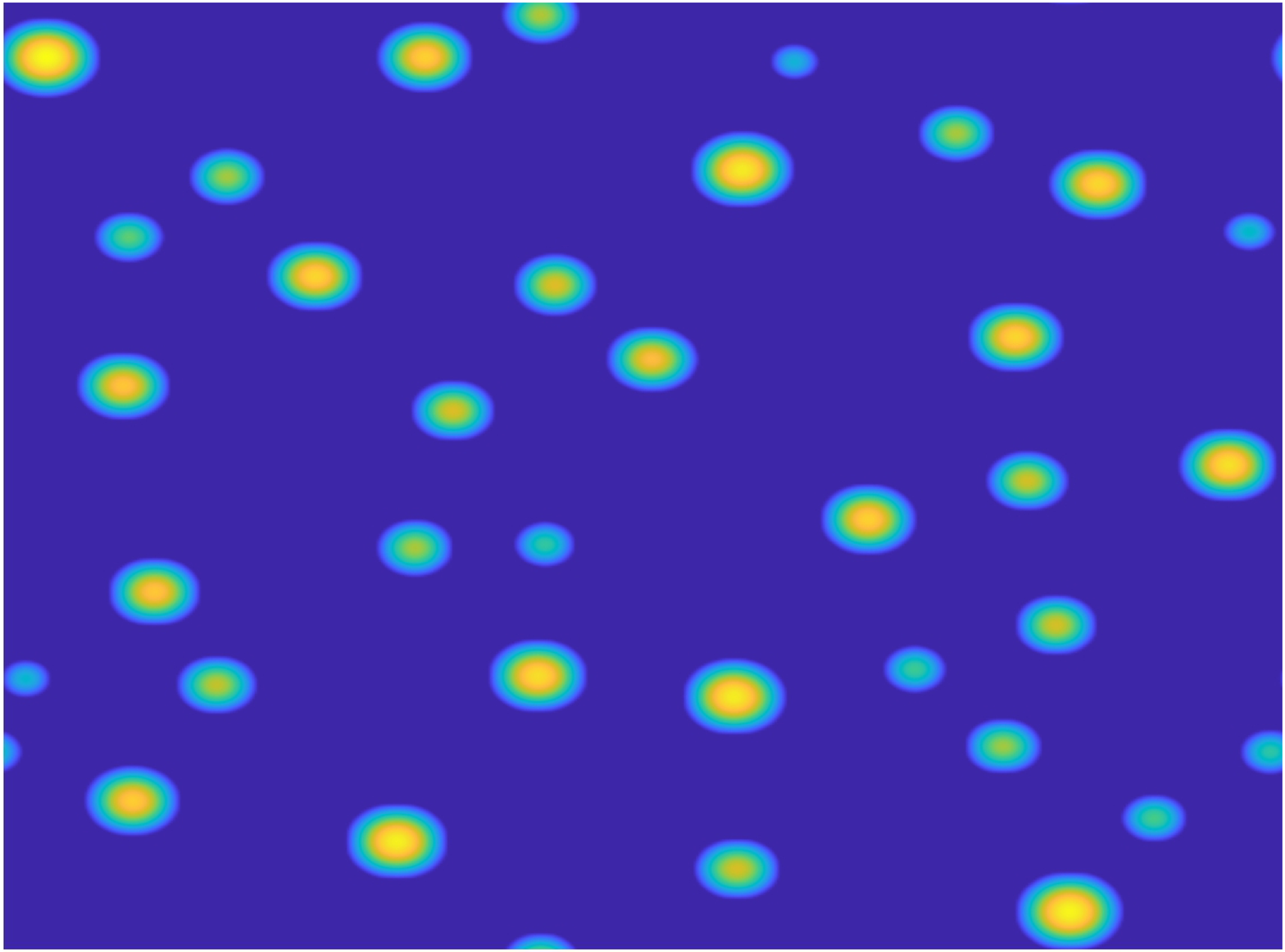}
			\includegraphics[height=0.48\textwidth,width=0.48\textwidth]{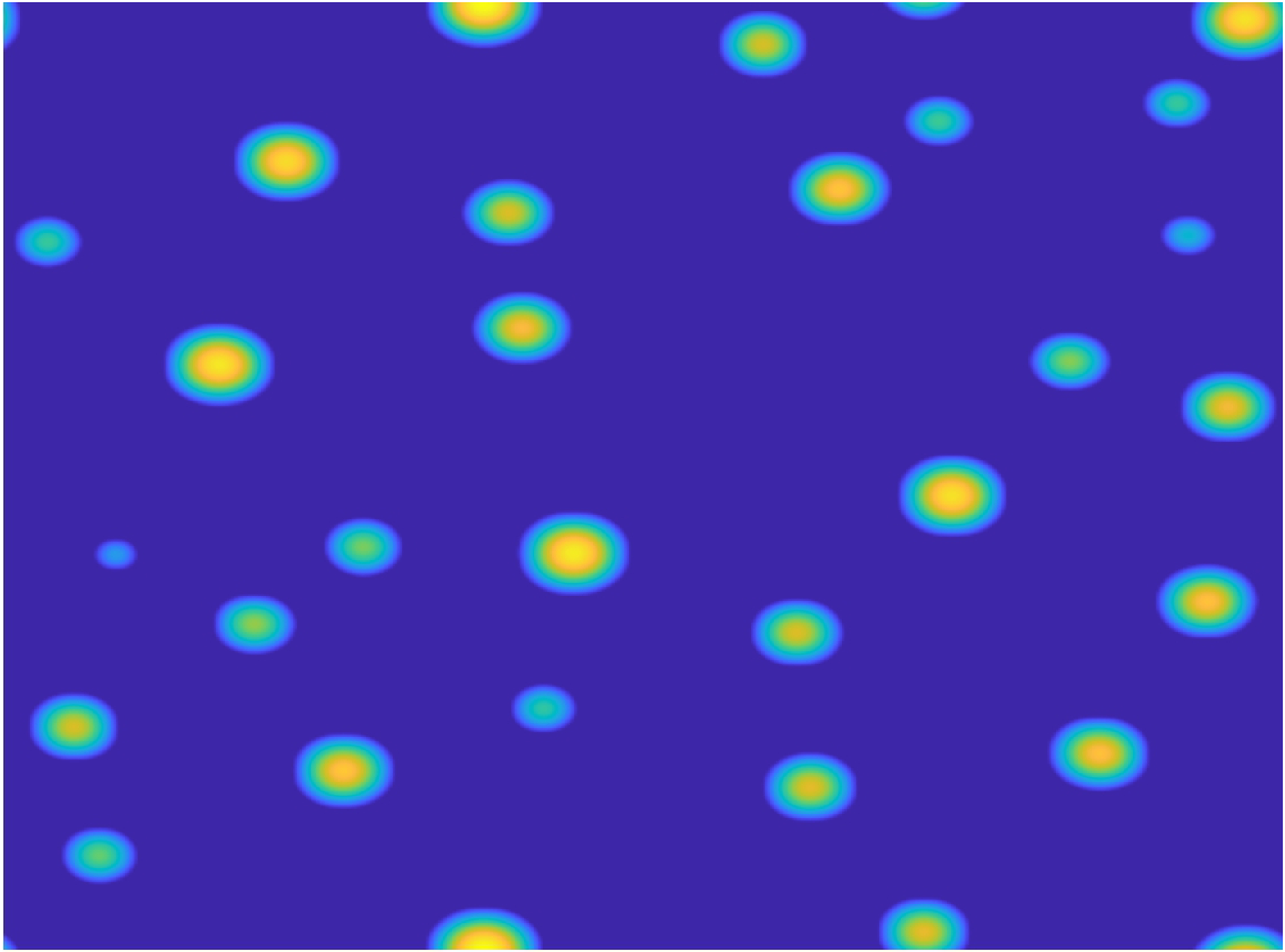}
			\caption*{$t=40, 60$}
		\end{subfigure}
		\begin{subfigure}{0.48\textwidth}
			\includegraphics[height=0.48\textwidth,width=0.48\textwidth]{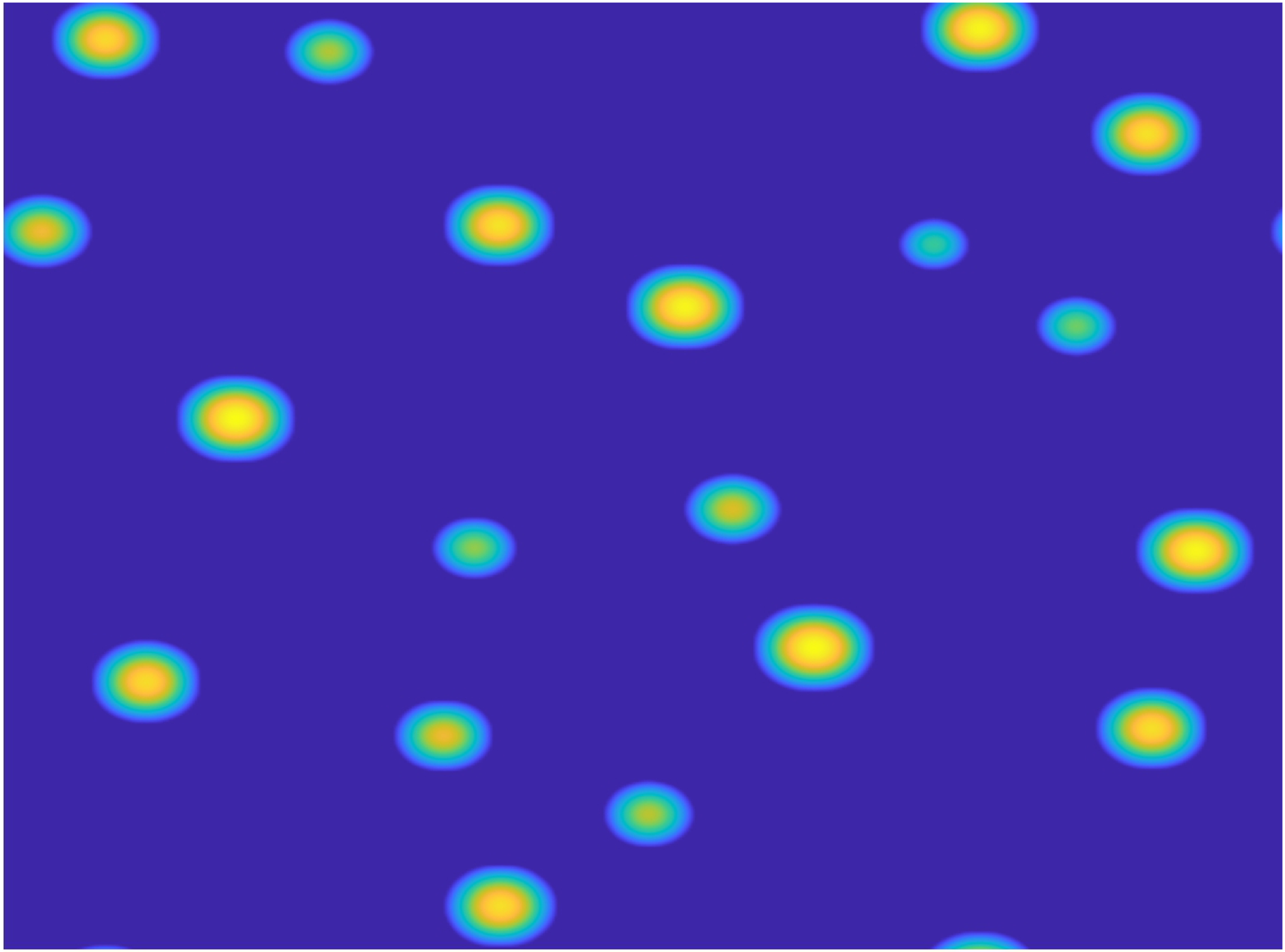}
			\includegraphics[height=0.48\textwidth,width=0.48\textwidth]{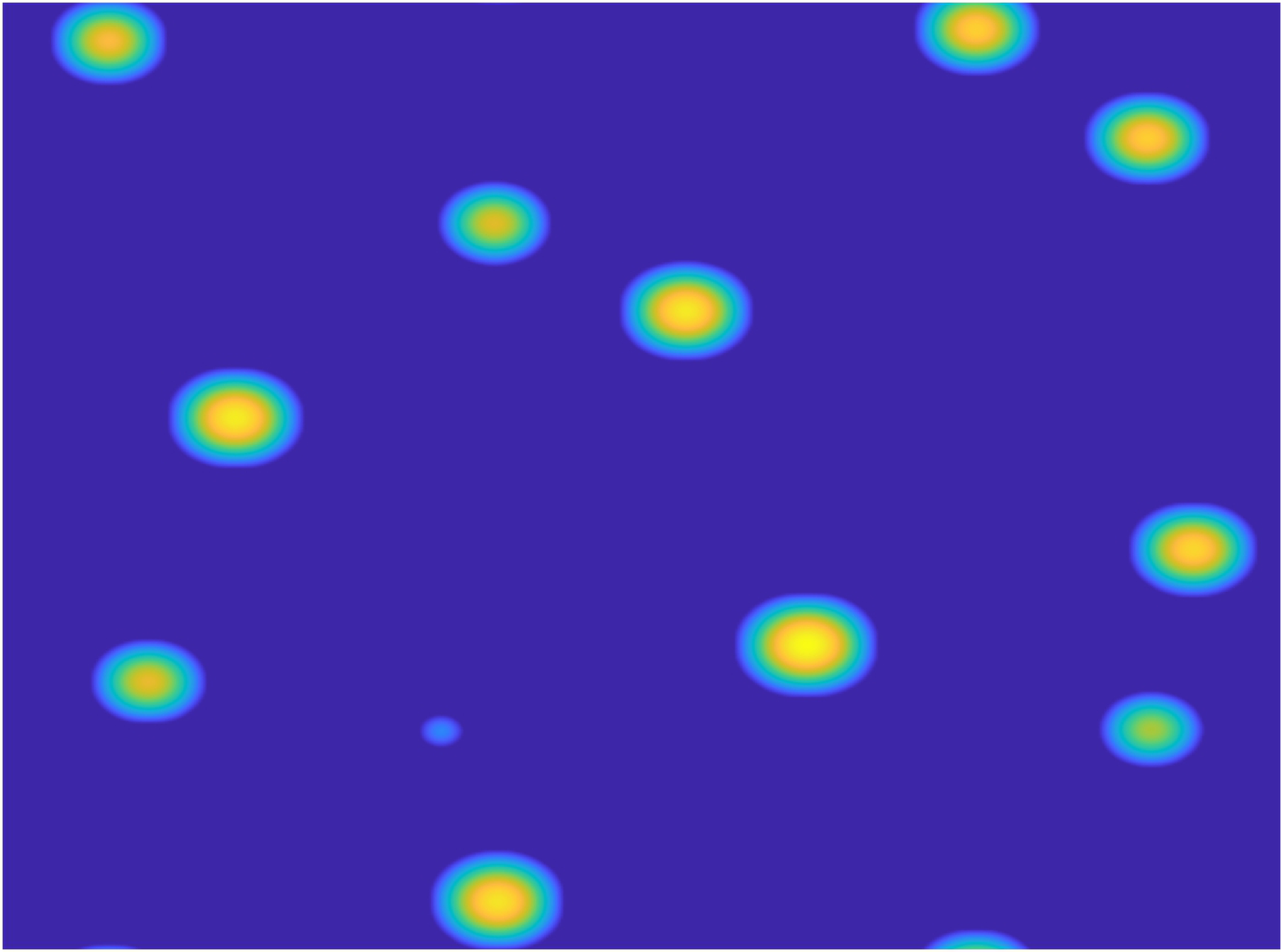}
			\caption*{$t=100, 200$}
		\end{subfigure}
		\begin{subfigure}{0.48\textwidth}
			\includegraphics[height=0.48\textwidth,width=0.48\textwidth]{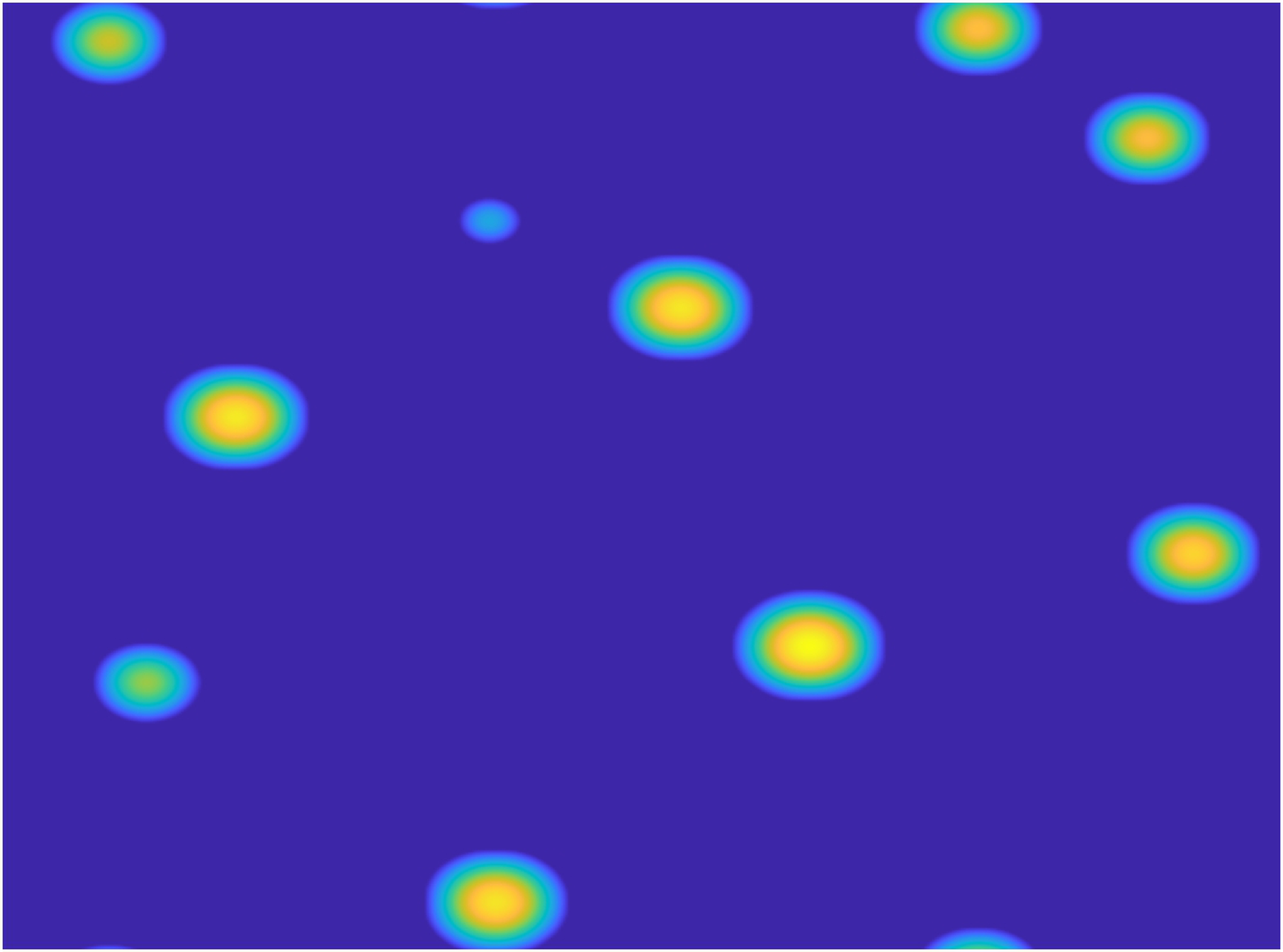}
			\includegraphics[height=0.48\textwidth,width=0.48\textwidth]{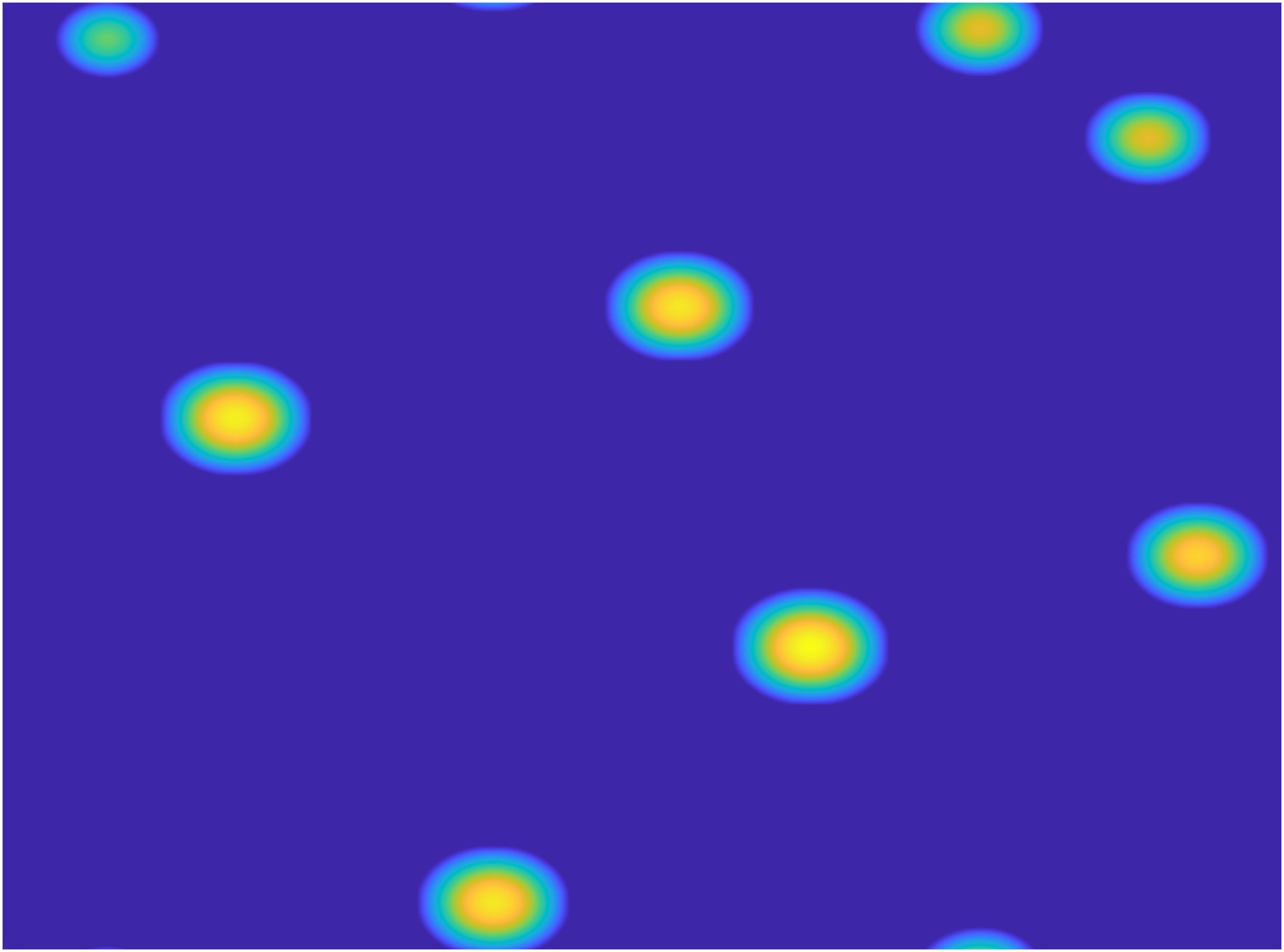}
			\caption*{$t=300, 400$}
		\end{subfigure}
		\begin{subfigure}{0.48\textwidth}
			\includegraphics[height=0.48\textwidth,width=0.48\textwidth]{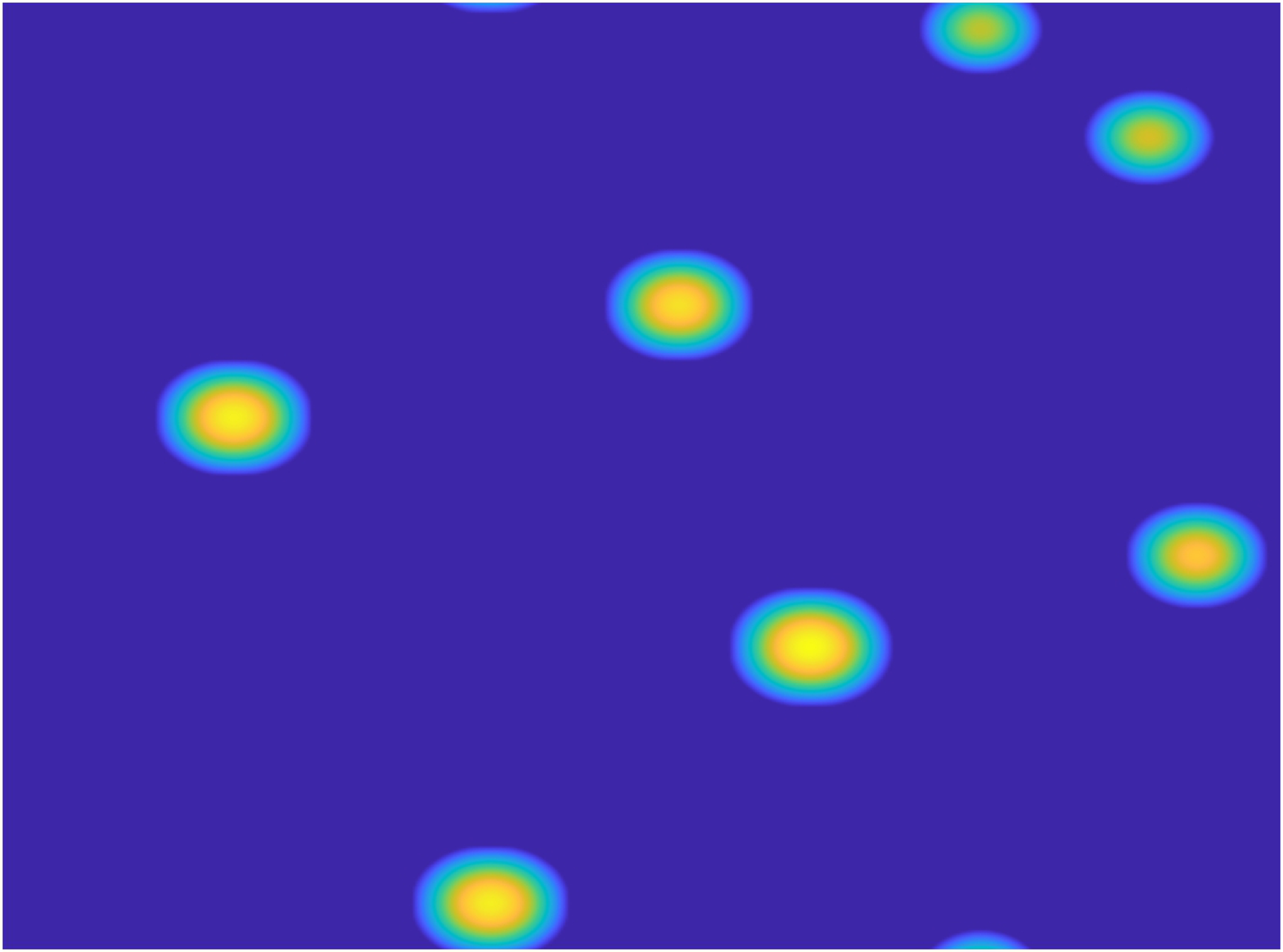}
			\includegraphics[height=0.48\textwidth,width=0.48\textwidth]{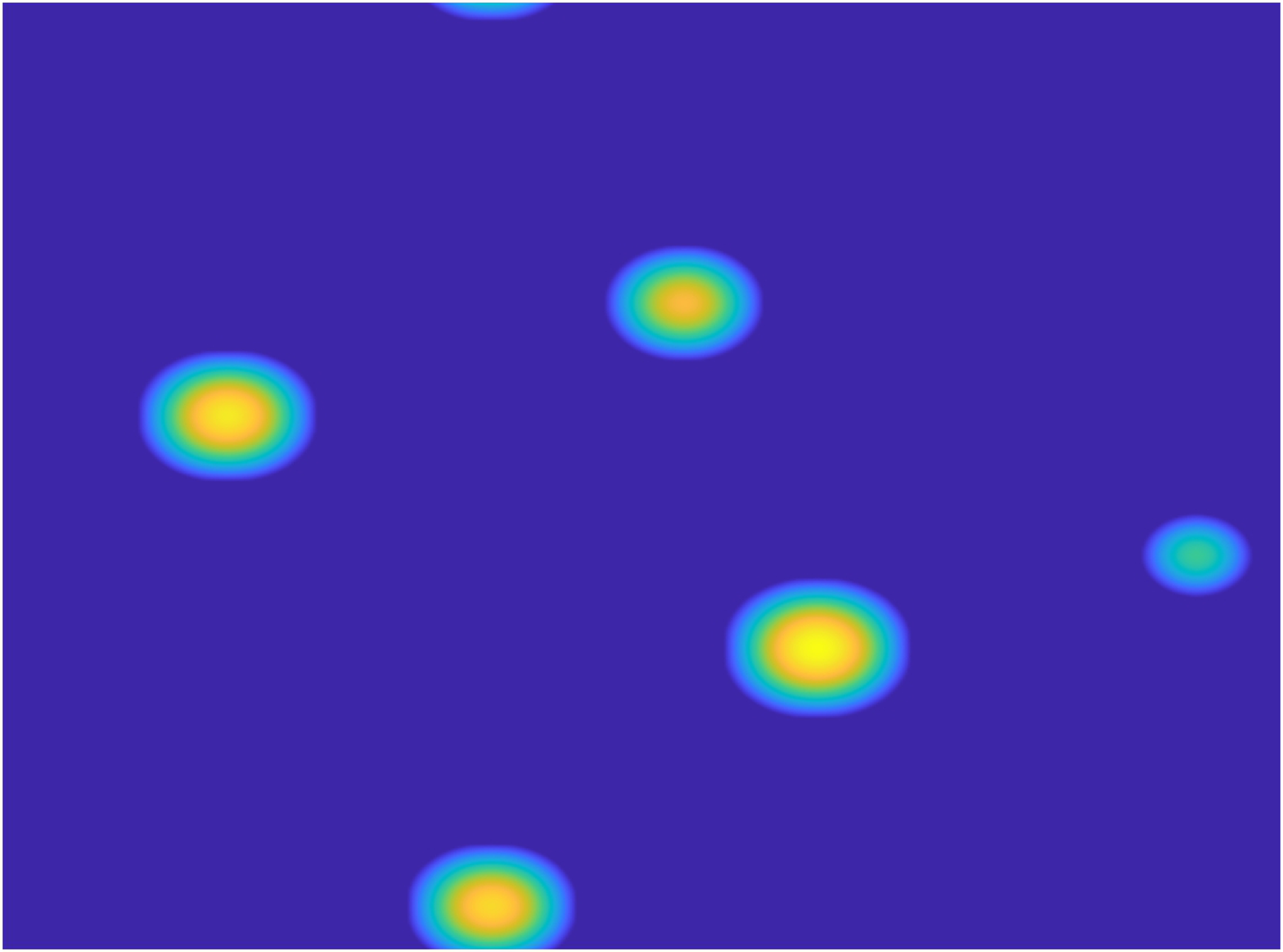}
			\caption*{$t=500, 900$}
		\end{subfigure}
		\begin{subfigure}{0.48\textwidth}
			\includegraphics[height=0.48\textwidth,width=0.48\textwidth]{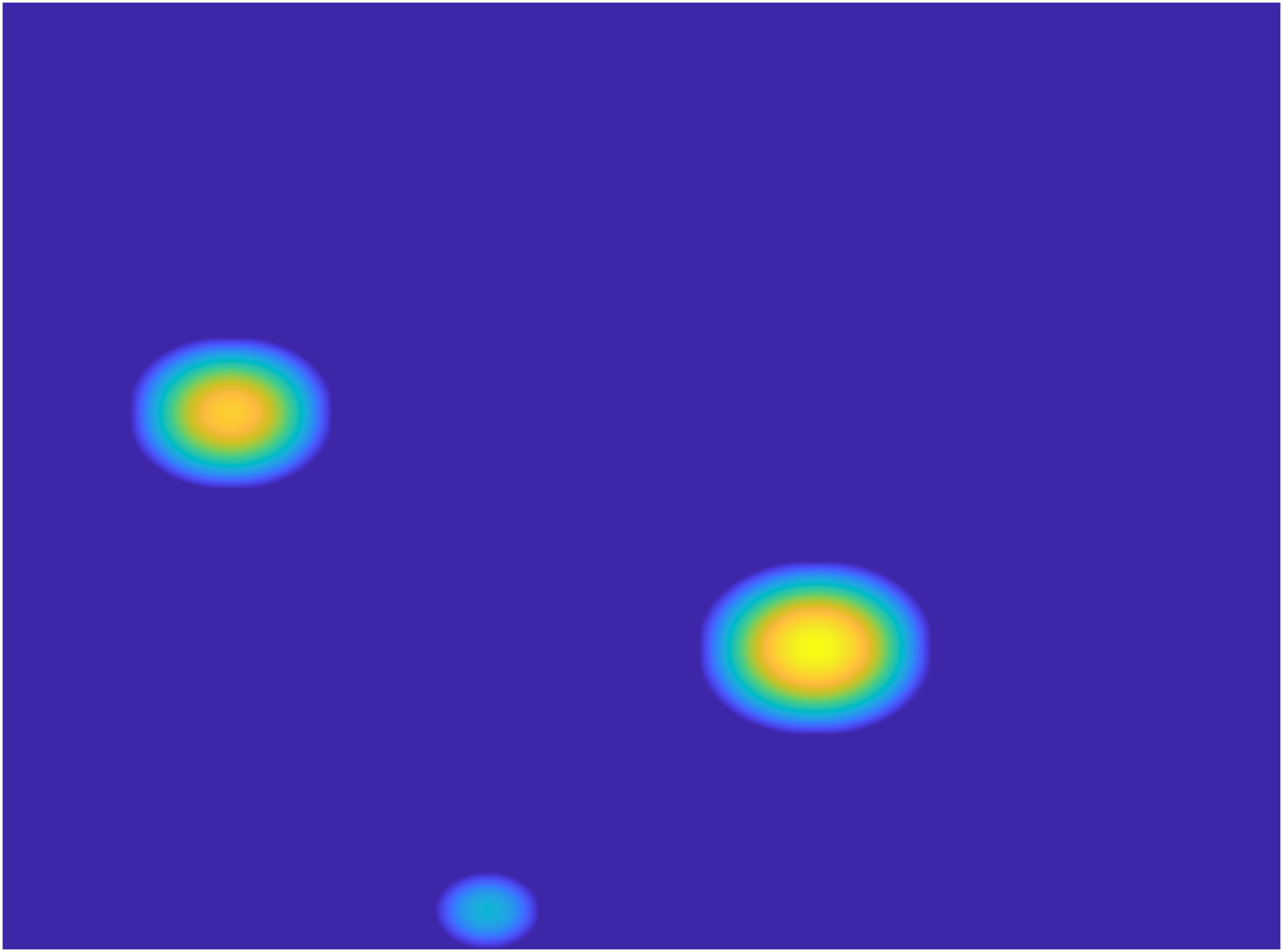}
			\includegraphics[height=0.48\textwidth,width=0.48\textwidth]{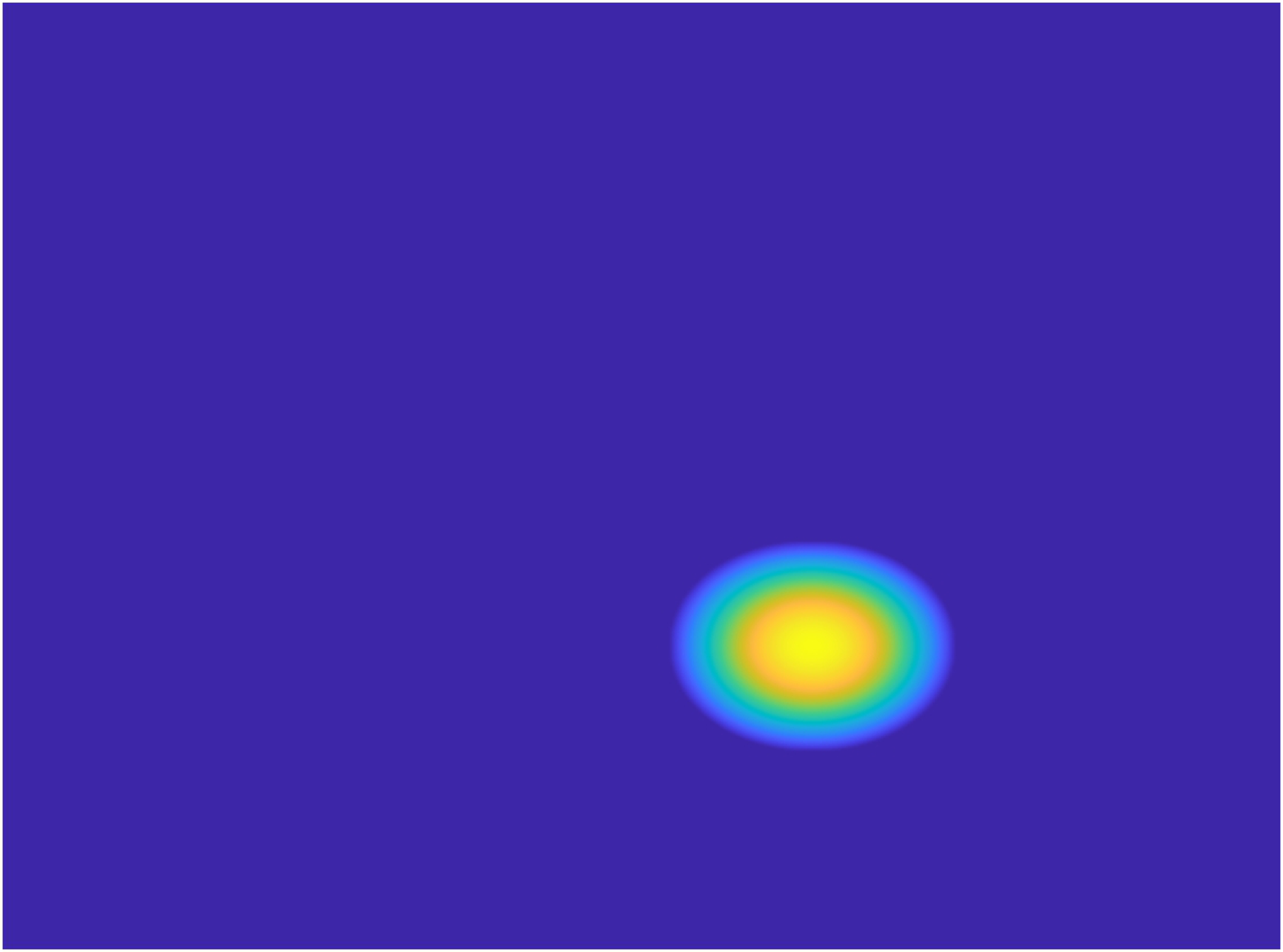}
			\caption*{$t=2000, 6000$}
		\end{subfigure}		
\caption{(Color online.) Snapshots of the computed height function $\phi$ at the indicated times for the parameters $L =12.8$, $\varepsilon = 0.02$.  Finally, there is a single droplet structure.}
		\label{fig3}
	\end{center}
\end{figure}

The long time characteristics of the solution, especially the energy decay rate, are of interest to material scientists. 
Figure~\ref{fig:energy evolution} presents the log-log plot for the energy versus time, with the given physical parameters. Recall that, at the space-discrete level, the energy, $F_h$ is defined via~\eqref{def: discrete energy}. The detailed scaling ``exponent" is obtained using least squares fits of the computed data up to time $t=100$.  A clear observation of the $a_e t^{b_e}$ scaling law can be made,  with $a_e = 59.8167$, $b_e=-0.1951$. For the Cahn-Hilliard flow with polynomial approximation of the double-well energy potential, various numerical experiments have indicated an approximately $t^{-1/3}$ energy dissipation law~\cite{cheng2019a, cheng16a}. For the droplet liquid film equation~\eqref{equation-AG}, with Lennard-Jones-type energy potential included in~\eqref{energy-AG-1}, this numerical evidence has implied a different energy dissipation scaling index, with $b_e = -0.1951$, in comparison with an approximate $t^{-1/3}$ scaling law for the standard Cahn-Hilliard model. 

	\begin{figure}
	\begin{center}
\includegraphics[width=3.0in]{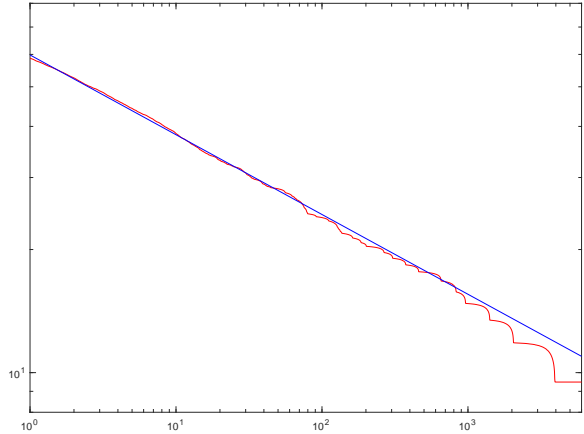}
	\end{center}
\caption{Log-log plot of the temporal evolution the energy $F_h$ for $\varepsilon=0.02$.  The energy decreases like $a_e t^{b_e}$ until saturation. The red lines represent the energy plot obtained by the simulations, while the straight lines are obtained by least squares approximations to the energy data.  The least squares fit is only taken for the linear part of the calculated data, only up to about time $t=100$. The fitted line has the form $a_e t^{b_e}$, with $a_e = 59.8167$, $b_e=-0.1951$.}
 	\label{fig:energy evolution}
 	\end{figure}
	
\begin{rem}
The coarsening dynamics problem usually is a long time process. To improve the computational efficiency, some adaptive time stepping strategies have been extensively applied in such a long-time simulation effort; our simulation has also used variable time step size method, as outlined above. There have been some theoretical works for the time-adaptive methods in the computation of coarsening process, such as the ones for the Cahn-Hilliard model~\cite{chen19c} and epitaxial thin film model~\cite{qiao2011}. The corresponding analysis for the time adaptive method applied to the droplet liquid film coarsening dynamics will be left to the future works.
\end{rem}

\begin{rem}
For the physical energy given by~\eqref{energy-AG-1}, the final steady state solution depends on $\varepsilon$ and the the mass average of $\phi$, which is a conserved quantity. It is well-known that, if the mass average is less than or equal to 1, the steady state solution turns out to be a trivial constant profile, due to its minimum energy constrained by the mass average. On the other hand, if the mass average is grater than 1, the time coarsening process and the final profile have interesting structures, as demonstrated in our reported numerical results. Under a circumstance that the initial data contain certain values near zero at certain region, the singular and convex nature of the Lennard-Jones-type energy potential ~\eqref{energy-AG-1} will push the solution away from the singular value of 0; in this case, if the mass average is greater than 1, one could still observe the interesting coarsening process, similar to our numerical simulation result.
\end{rem}



\section{Concluding remarks}  \label{sec:conclusion}

Two finite difference numerical schemes, including both the first and second order accurate algorithms, are proposed and analyzed for the droplet liquid film model, in which a singular Leonard-Jones energy potential is involved. In the first order scheme, the convex potential and the surface diffusion terms are implicitly, while the concave potential term is updated explicitly. A theoretical justification of the unique solvability and positivity-preserving property is provided, so that a singularity is avoided in the scheme. Such an analysis is based on the subtle fact that the numerical solution is equivalent to a minimization of a convex numerical energy, and the singular nature of the Leonard-Jones potential term around the value of 0 prevents the numerical solution reaching such singular value. The energy stability of the numerical schemes comes from the convex-concave decomposition of the physical energy. In the second order scheme, an alternate convex-concave decomposition is formulated so that the concave part corresponds to a quadratic energy. In turn, the BDF temporal stencil is applied, the combined Leonard-Jones potential term is treated implicitly, the concave expansive part the is approximated by a second order Adams-Bashforth explicit extrapolation, and an artificial Douglas-Dupont regularization term is added to ensure the energy stability. The unique solvability and the positivity-preserving property for the second order scheme could be similarly established, in which the convexity of the $\phi^{-8}$ term has played an important role. The preconditioned steepest descent (PSD) iteration solver has been applied to implement the nonlinear numerical schemes, due to the fact the implicit part turns out to be the gradient of a strictly convex energy functional. A few numerical results are also presented in this article, which demonstrates the robustness of the proposed numerical scheme.

	\section*{Acknowledgements}
This work is supported in part by the National Science Foundation (USA) grants NSF DMS-2012669 (C.~Wang), and NSF DMS-1719854, DMS-2012634
(S.~Wise), the National Natural Science Foundation of China (NSFC) 11871105, 11571045 and Science Challenge Project TZ2018002 (Z. Zhang).

	\bibliographystyle{plain}
	\bibliography{revision1}

\begin{thebibliography}{10}

\bibitem{Alexander1999}
O.~Alexander and S.G. Bankoff.
\newblock Dewetting of a heated surface by an evaporating liquid film under
  conjoining/disjoining pressures.
\newblock {\em J. Colloid Interface Sci.}, 218:152--166, 1999.

\bibitem{Alexander1997}
O.~Alexander, S.H. Davis, and S.G. Bankoff.
\newblock Long-scale evolution of thin liquid films.
\newblock {\em Rev. Mod. Phys.}, 69:931--980, 1997.

\bibitem{Anna1997}
K.S. Anna, Z.Y. Zhang, and J.F. Wendelken.
\newblock Island diffusion and coarsening on metal (100) surfaces.
\newblock {\em Phys. Rev. Lett.}, 79:3210--3213, 1997.

\bibitem{Becker2003}
J.~Becker, G.~Grun, R.~Seeman, H.~Mantz, K.~Jacobs, K.~R. Mecke, and
  R.~Blossey.
\newblock Complex dewetting scenarios captured by thin-film models.
\newblock {\em Nature Materials}, 2:59--63, 2003.

\bibitem{Bertozzi2001}
A.L. Bertozzi, G.~Grun, and T.~P. Witelski.
\newblock Dewetting films: bifurcations and concentrations, institute of
  physics publishing.
\newblock {\em Institute of physics publishing}, 14:1569--1592, 2001.

\bibitem{chen20c}
W.~Chen, C.~Wang, S.~Wang, X.~Wang, and S.~Wise.
\newblock Energy stable numerical schemes for a ternary {Cahn-Hilliard} system.
\newblock {\em J. Sci. Comput.}, 84:27, 2020.

\bibitem{chen19b}
W.~Chen, C.~Wang, X.~Wang, and S.M. Wise.
\newblock Positivity-preserving, energy stable numerical schemes for the
  {Cahn-Hilliard} equation with logarithmic potential.
\newblock {\em J. Comput. Phys.: X}, 3:100031, 2019.

\bibitem{chen19c}
W.~Chen, X.~Wang, Y.~Yan, and Z.~Zhang.
\newblock A second order bdf numerical scheme with variable steps for the
  {Cahn-Hilliard} equation.
\newblock {\em SIAM J. Numer. Anal.}, 57:495--525, 2019.

\bibitem{cheng2019a}
K.~Cheng, W.~Feng, C.~Wang, and S.M. Wise.
\newblock An energy stable fourth order finite difference scheme for the
  {Cahn-Hilliard} equation.
\newblock {\em J. Comput. Appl. Math.}, 362:574--595, 2019.

\bibitem{cheng2019d}
K.~Cheng, C.~Wang, and S.M. Wise.
\newblock An energy stable {Fourier} pseudo-spectral numerical scheme for the
  square phase field crystal equation.
\newblock {\em Commun. Comput. Phys.}, 26:1335--1364, 2019.

\bibitem{cheng16a}
K.~Cheng, C.~Wang, S.M. Wise, and X.~Yue.
\newblock A second-order, weakly energy-stable pseudo-spectral scheme for the
  {Cahn-Hilliard} equation and its solution by the homogeneous linear iteration
  method.
\newblock {\em J. Sci. Comput.}, 69:1083--1114, 2016.

\bibitem{miranville11}
L.~Cherfils, A.~Miranville, and S.~Zelik.
\newblock The {Cahn-Hilliard} equation with logarithmic potentials.
\newblock {\em Milan J. Math.}, 79:561--596, 2011.

\bibitem{dong19a}
L.~Dong, C.~Wang, H.~Zhang, and Z.~Zhang.
\newblock A positivity-preserving, energy stable and convergent numerical
  scheme for the {Cahn-Hilliard} equation with a {Flory-Huggins-deGennes}
  energy.
\newblock {\em Commun. Math. Sci.}, 17:921--939, 2019.

\bibitem{dong20a}
L.~Dong, C.~Wang, H.~Zhang, and Z.~Zhang.
\newblock A positivity-preserving second-order {BDF} scheme for the
  {Cahn-Hilliard} equation with variable interfacial parameters.
\newblock {\em Commun. Comput. Phys.}, 28:967--998, 2020.

\bibitem{Du2019}
Q.~Du, L.~Ju, X.~Li, and Z.~Qiao.
\newblock Maximum principle preserving exponential time differencing schemes
  for the nonlocal {Allen-Cahn} equation.
\newblock {\em SIAM J. Numer. Anal.}, 57:875–898, 2019.

\bibitem{Du2020a}
Q.~Du, L.~Ju, X.~Li, and Z.~Qiao.
\newblock Maximum bound principles for a class pf semi-linear parabolic
  equations and exponential time differencing schemes.
\newblock {\em SIAM Rev.}, 2020.
\newblock Accepted and in press: https://arxiv.org/abs/2015.114.65.

\bibitem{fengW18a}
W.~Feng, Z.~Guan, J.S. Lowengrub, C.~Wang, S.M. Wise, and Y.~Chen.
\newblock A uniquely solvable, energy stable numerical scheme for the
  functionalized {Cahn-Hilliard} equation and its convergence analysis.
\newblock {\em J. Sci. Comput.}, 76(3):1938--1967, 2018.

\bibitem{feng2017}
W.~Feng, A.J. Salgado, C.~Wang, and S.M. Wise.
\newblock Preconditioned steepest descent methods for some nonlinear elliptic
  equations involving p-{Laplacian} terms.
\newblock {\em J. Comput. Phys.}, 334:45--67, 2017.

\bibitem{fengW18b}
W.~Feng, C.~Wang, S.M. Wise, and Z.~Zhang.
\newblock A second-order energy stable {Backward Differentiation Formula}
  method for the epitaxial thin film equation with slope selection.
\newblock {\em Numer. Methods Partial Differ. Equ.}, 34(6):1975--2007, 2018.

\bibitem{Giorgini17a}
A.~Giorgini, M.~Grasselli, and A.~Miranville.
\newblock The {Cahn-Hiliard-Oono} equation with singular potential.
\newblock {\em Math. Models Methods Appl. Sci.}, 27(13):2485--2510, 2017.

\bibitem{Glasner2008}
K.B. Glasner.
\newblock Ostwald ripening in thin film equations.
\newblock {\em SIAM J. Appl. Math.}, 69:473--493, 2008.

\bibitem{Glasner2003}
K.B. Glasner and T.~P. Witelski.
\newblock Coarsening dynamics of dewetting films.
\newblock {\em Phys. Rev. E.}, 67:1--12, 2003.

\bibitem{guo16}
J.~Guo, C.~Wang, S.M. Wise, and X.~Yue.
\newblock An {$H^2$} convergence of a second-order convex-splitting, finite
  difference scheme for the three-dimensional {Cahn-Hilliard} equation.
\newblock {\em Commun. Math. Sci.}, 14:489--515, 2016.

\bibitem{Hao2020}
Y.~Hao, Q.~Huang, and C.~Wang.
\newblock A third order {BDF} energy stable linear scheme for the
  no-slope-selection thin film model.
\newblock {\em Commun. Comput. Phys.}, 2020.
\newblock Accepted and in press.

\bibitem{Israelachvili1992}
J.N. Israelachvili.
\newblock {\em Intermolecular and Surface Forces}.
\newblock Academic Press, New York, NY, 1992.

\bibitem{Kim2006}
J.~Kim.
\newblock Numerical simulations of phase separation dynamics in a
  water-oil-surfactant system.
\newblock {\em J. Colloid Interface Sci.}, 303:272--279, 2006.

\bibitem{kohn02}
R.V. Kohn and F.~Otto.
\newblock Upper bound on coarsening rate.
\newblock {\em Commun. Math. Phys.}, 229:375--395, 2002.

\bibitem{Kohn2004}
R.V. Kohn and X.D. Yan.
\newblock Coarsening rates for models of multicomponent phase separation.
\newblock {\em Interfaces and Free Boundaries}, 6:135--149, 2004.

\bibitem{LiD2017}
D.~Li and Z.~Qiao.
\newblock On second order semi-implicit {Fourier spectral} methods for {2D
  Cahn-Hilliard} equations.
\newblock {\em J. Sci. Comput.}, 70:301--341, 2017.

\bibitem{LiD2017b}
D.~Li and Z.~Qiao.
\newblock On the stabilization size of semi-implicit {Fourier-spectral} methods
  for {3D Cahn-Hilliard} equations.
\newblock {\em Commun. Math. Sci.}, 15:1489--1506, 2017.

\bibitem{LiW18}
W.~Li, W.~Chen, C.~Wang, Y.~Yan, and R.~He.
\newblock A second order energy stable linear scheme for a thin film model
  without slope selection.
\newblock {\em J. Sci. Comput.}, 76(3):1905--1937, 2018.

\bibitem{Limary2002}
R.~Limary and P.~F. Green.
\newblock Late-stage coarsening of an unstable structured liquid film.
\newblock {\em Phys. Rev. E.}, 66:1--6, 2002.

\bibitem{Limary2003}
R.~Limary and P.~F. Green.
\newblock Dynamics of droplets on the surface of a structured fluid film:
  {L}ate-stage coarsening.
\newblock {\em Langmuir}, 19:2419--2424, 2003.

\bibitem{Meng2020}
X.~Meng, Z.~Qiao, C.~Wang, and Z.~Zhang.
\newblock Artificial regularization parameter analysis for the
  no-slope-selection epitaxial thin film model.
\newblock {\em CSIAM Trans. Appl. Math.}, 1:441--462, 2020.

\bibitem{miranville12}
A.~Miranville.
\newblock On a phase-field model with a logarithmic nonlinearity.
\newblock {\em Appl. Math.}, 57:215--229, 2012.

\bibitem{Oron97}
A.~Oron, S.~H. Davis, and S.~G. Bankoff.
\newblock Long-scale evolution of thin liquid films.
\newblock {\em Rev. Mod. Phys.}, 69:932--977, 1997.

\bibitem{otto06}
F.~Otto, T.~Rump, and D.~Slepcev.
\newblock Coarsening rates for a droplet model: {R}igorous upper bounds.
\newblock {\em SIAM J. Math. Anal.}, 38(2):503, 2006.

\bibitem{Pbgo1989}
R.L. Pbgo and A.M. Stuart.
\newblock Front migration in the nonlinear {Cahn-Hilliard} equation.
\newblock {\em Proc. R. Soc. Lond.}, 422:261--278, 1989.

\bibitem{Pismen2004}
L.M. Pismen.
\newblock Spinodal dewetting in a volatile liquid film.
\newblock {\em Phys. Rev. E.}, 70:1--9, 2004.

\bibitem{qiao2011}
Z.~Qiao, Z.~Zhang, and T.~Tang.
\newblock An adaptive time-stepping strategy for the molecular beam epitaxy
  models.
\newblock {\em SIAM J. Sci. Comput.}, 33:1395--1414, 2011.

\bibitem{Ravati2013}
S.~Ravati and B.D. Favis.
\newblock Interfacial coarsening of ternary polymer blends with partial and
  complete wetting structures.
\newblock {\em Polymer}, 54:6739--6751, 2013.

\bibitem{Thiele2001}
U.~Thiele, M.~G. Velarde, and K.~NeuKer.
\newblock Dewetting: film rupture by nucleation in the spinodal regime.
\newblock {\em Phys. Rev. Lett.}, 87:1--4, 2001.

\bibitem{wang11a}
C.~Wang and S.M. Wise.
\newblock An energy stable and convergent finite-difference scheme for the
  modified phase field crystal equation.
\newblock {\em SIAM J. Numer. Anal.}, 49:945--969, 2011.

\bibitem{wise10}
S.M. Wise.
\newblock Unconditionally stable finite difference, nonlinear multigrid
  simulation of the {Cahn-Hilliard-Hele-Shaw} system of equations.
\newblock {\em J. Sci. Comput.}, 44:38--68, 2010.

\bibitem{wise09a}
S.M. Wise, C.~Wang, and J.~Lowengrub.
\newblock An energy stable and convergent finite-difference scheme for the
  phase field crystal equation.
\newblock {\em SIAM J. Numer. Anal.}, 47:2269--2288, 2009.

\bibitem{yan18}
Y.~Yan, W.~Chen, C.~Wang, and S.M. Wise.
\newblock A second-order energy stable {BDF} numerical scheme for the
  {Cahn-Hilliard} equation.
\newblock {\em Commun. Comput. Phys.}, 23:572--602, 2018.

\end{thebibliography}

\end{document}